\documentclass[11pt,a4paper,reqno]{article} 

\usepackage{anysize}
\marginsize{3.5cm}{3.5cm}{2.2cm}{2.2cm}

\usepackage{hyperref}
\hypersetup{
    colorlinks=true,       
    citecolor=blue,        
}

\usepackage[english]{babel}
\usepackage{amsmath}
\usepackage{amsfonts,amssymb}
\usepackage{enumerate}
\usepackage{mathrsfs}
\usepackage{amsthm}
\usepackage{tikz}
\usetikzlibrary{arrows}
\usepackage{mathrsfs}
\usepackage{caption}

\theoremstyle{plain}
\newtheorem{theo}{Theorem}[section]
\newtheorem*{theo*}{Theorem}
\newtheorem*{main*}{Main result}
\newtheorem{prop}[theo]{Proposition}
\newtheorem{lemm}[theo]{Lemma}
\newtheorem{coro}[theo]{Corollary}
\newtheorem{defi}[theo]{Definition}
\theoremstyle{definition}
\newtheorem{rema}[theo]{Remark}
\newtheorem*{rema*}{Remark}
\newtheorem{nota}[theo]{Notation}
\newtheorem*{nota*}{Notation}
\DeclareMathOperator{\cn}{div}
\DeclareMathOperator{\dist}{dist}
\def\cnx{\cn_x}
\def\be{\begin{equation}}
\def\ee{\end{equation}}
\def\e{\eqref}
\def\pair{e}

\def\defn{\mathrel{:=}}
\def\eps{\varepsilon}
\def\la{\left\vert}
\def\lA{\left\Vert}
\def\le{\leq}
\def\les{\lesssim}
\def\mez{\frac{1}{2}}
\def\ra{\right\vert}
\def\rA{\right\Vert}
\def\tdm{\frac{3}{2}}

\def\uq{\frac{1}{4}}

\def\xN{\mathbb{N}}
\def\xR{\mathbb{R}}

\def\xT{\mathbb{T}}
\def\xZ{\mathbb{Z}}
\def\px{\partial_x}
\def\py{\partial_y}
\def\pz{\partial_z}
\def\va{\varphi}

\def\poh{x\cdot\nabla}

\def\Qr{\mathscr{R}}
\title{Boundary observability of gravity water waves}
\author{Thomas Alazard \\ CNRS \& \'Ecole normale sup\'erieure}
\date{\empty}

\def\notina[#1]#2{\begingroup\def\thefootnote{\fnsymbol{footnote}}\footnote[#1]{#2}\endgroup}

\begin{document}

\maketitle

\begin{abstract}
Consider a three-dimensional fluid in a rectangular tank, bounded by a flat bottom, vertical walls 
and a free surface 
evolving under the influence of gravity. We prove that one can estimate its energy 
by looking only at the motion of the points of contact between the free surface 
and the vertical walls. The proof relies on the multiplier technique, 
the Craig-Sulem-Zakharov 
formulation of the water-wave problem, a Pohozaev identity for the 
Dirichlet to Neumann operator, previous results about the Cauchy problem and computations 
inspired by the analysis done by Benjamin and Olver of the conservation laws for water waves. 
\end{abstract}

\section{Introduction}

Consider surface waves over an incompressible liquid, evolving under the influence of gravity, 
\notina[0]{This work is 
partly supported by the grant ``ANA\'E'' ANR-13-BS01-0010-03.} 
in the case where the fluid is located inside a fixed rectangular tank 
$\Qr$ of the form 
$\Qr= Q\times [-h,+\infty)$ where 
$Q=[0,L_1]\times [0,L_2]$ and $h$ is a positive constant. 
At time~$t$, the fluid domain $\Omega(t)$ is 
given by
\be\label{n1}
\Omega(t)=\left\{\, (x,y)\,:\, x\in Q,~-h\le y\le \eta(t,x)\,\right\},%\\
\ee
where $x=(x_1,x_2)$ (resp.\ $y$) is the horizontal (resp.\ vertical) space variable. 
The equations which dictate the motion are the incompressible Euler equations with free surface. 
This is a system of two nonlinear equations: 
the incompressible Euler equation for the velocity potential $\phi\colon \Omega\rightarrow \xR$ (so that the velocity is 
$v=\nabla_{x,y}\phi$) and a 
kinematic equation for $\eta$ which states that 
the free surface moves with the fluid. 
The energy, which is the 
sum of the potential energy and the kinetic energy, is conserved:
\be\label{n2}
\frac{d\mathcal{H}}{dt}=0\quad \text{with}\quad \mathcal{H}=\frac{g}{2}\int_Q\eta^2(t,x)\, dx
+\mez\iint_{\Omega(t)}\la \nabla_{x,y}\phi(t,x,y)\ra^2\, dx dy,
\ee
where $g$ is the acceleration of gravity. 
This paper is devoted to the analysis of the following question: 
is-it possible to estimate the energy $\mathcal{H}$ of gravity water waves by looking only at the 
motion of some of the curves of contact between 
the free surface and the vertical walls? 
From the point of view of control theory, this is the question of boundary observability 
of gravity water waves.

\subsection{The water-wave equations}\label{S:11}

Hereafter we use the notations 
$\nabla=(\partial_{x_1},\partial_{x_2})$, 
$\nabla_{x,y}=(\nabla_x,\partial_y)$ and $\Delta_{x,y}=\Delta_x+\partial_y^2$. 
As already mention, the velocity field $v\colon \Omega \rightarrow \xR^{3}$ is given by $v=\nabla_{x,y} \phi$ 
for some velocity potential~$\phi\colon \Omega\rightarrow \xR$ satisfying 
\begin{equation}\label{n5}
\Delta_{x,y}\phi=0,\quad 
\partial_{t} \phi +\mez \la \nabla_{x,y}\phi\ra^2 + P +g y = 0,
\end{equation}
where 
$P$ is the pressure term. 
The water-wave equations are then given by several boundary conditions. Firstly, the velocity 
satisfies the solid wall boundary condition, so that $\partial_n\phi= 0$ 
on the boundary of the tank $\partial\mathscr{R}\cap \partial\Omega$; where 
$\partial_n$ denotes the outward normal derivative. 
This implies that 
\begin{alignat}{3}
&\partial_{x_1} \phi =0 \quad&&\text{for}\quad && x_1=0 \text{ or }x_1=L_1,\label{n6}\\
&\partial_{x_2} \phi =0 \quad&&\text{for}\quad && x_2=0 \text{ or }x_2=L_2,\label{n6b}\\
&\py \phi =0 \quad&&\text{for}\quad && y=-h.\label{n7}
\end{alignat}
The problem is then given by two boundary conditions on the free surface $\{y=\eta\}$:
\begin{align}
&\partial_{t} \eta = \sqrt{1+|\nabla\eta|^2}\, \partial_n \phi \arrowvert_{y=\eta}=\py\phi(t,x,\eta)-\nabla\eta(t,x)\cdot\nabla\phi(t,x,\eta), \label{n8}\\
& P\arrowvert_{y=\eta}=0.\label{n9}
\end{align}
The equation \e{n8} implies that 
the free surface moves with the fluid. 
The condition $P(t,x,\eta)=0$ is a dynamic condition 
that expresses a balance of forces across the free surface. 

We also assume that the free surface must intersect the vertical walls orthogonally:
\be\label{n9a}
\begin{aligned}
&\partial_{x_1}\eta=0\quad\text{for}\quad x_1=0\text{ or }x_1=L_1,\\
&\partial_{x_2}\eta=0\quad\text{for}\quad x_2=0\text{ or }x_2=L_2.
\end{aligned}
\ee
It is proved in \cite{ABZ4} 
that \e{n9a}Ê
always holds for smooth enough solutions (we elaborate on that fact in Section~\S\ref{S:canal}). 
Eventually we assume that
\be\label{n9b}
\eta\ge -\frac{h}{2},\quad \int_Q\eta \, dx=0.
\ee
One can always assume that the mean value vanishes since it is a conserved quantity. 

\subsection{Boundary observability of water waves}

 There are very few articles about the possible applications of control theory to the study of water waves. 
The first results are due to Reid and Russell \cite{ReidRussell1985} and Reid~\cite{Reid1986,Reid1995} who studied 
the linearized equations at the origin. 
Alazard, Baldi and Han-Kwan initiated in \cite{ABHK} the study of the control 
of the nonlinear equations. 
The analysis in \cite{ABHK}Ê
and the one in this paper 
rely on completely different tools. This is because of the following differences: 
the article \cite{ABHK} addresses the internal control problem for the two-dimensional equations with surface tension, 
while here we consider the boundary observability for the 2D and 3D problems without 
surface tension. In \cite{ABHK}, the analysis used in an essential way the 
infinite speed of propagation of gravity-capillary waves, which is the property that an harmonic with frequency $\la\xi\ra$ travels at a speed proportional to $\sqrt{\la\xi\ra}$. 
By contrast, for gravity waves, an harmonic with frequency $\la\xi\ra$ travels at a speed proportional to $1/\sqrt{\la\xi\ra}$. 
Loosely speaking, 
in this paper we study 
the observability problem for low and medium frequency gravity waves, while in \cite{ABHK} the controllability of gravity-capillary waves 
is deduced from an observation of high frequencies. 
With regards to the proofs, the paper 
\cite{ABHK} relies on {\em microlocal} analysis while here 
the proof will rely on the study of {\em global} quantities. 
More precisely, in \cite{ABHK} the analysis combines 
Ingham type inequalities, 
paradifferential calculus and other techniques used to study quasi-linear problems. By contrast, in this paper, 
the analysis will be based on the multiplier method and on various exact identities which are studied as conservations laws. 
In this direction, let us mention the recent article by Biccari~\cite{Biccari} which introduced the use of the multiplier method to analyze the interior controllability problem 
for the fractional Schr\"odinger equation $i\partial_tu+(-\Delta)^su=0$ with $s\ge 1/2$ in a $C^{1,1}$ bounded domain $Q$ with Dirichlet boundary condition\footnote{ 
One cannot apply this result to study the observability of gravity water waves. Indeed, 
$i)$ we are concerned with a nonlinear problem and $ii)$ %the Dirichlet boundary condition is not convenient for our 
%problem; $iii)$ 
even the linearized problem involves a different setting since it can be written under the form $i\partial_tu+(-\Delta)^su=0$ with 
$s=1/4$ (and hence the assumption $s\ge 1/2$ does not hold). The key point is that, 
if $s<1/2$, then high frequency waves propagate at a speed which goes to $0$ when $|\xi|$ goes to $+\infty$.}.

We do not assume that the reader is familiar with control theory and 
begin by recalling what is the multiplier method in the simplest case. 
Consider the one dimensional linear wave equation with Dirichlet boundary condition:
\be\label{n3a}
\partial_t^2u-\partial_x^2u=0,\quad u(t,0)=u(t,1)=0.
\ee
{\em Multiply} the equation by $x\partial_x u$ and integrate by parts, to obtain
\be\label{n3}
\mez \int_0^T (\partial_x u(t,1))^2\, dt=
\int_0^1(\partial_t u)(x\partial_x u)\, dx\Big\arrowvert_{0}^T
+\mez \iint_{S}\big[(\partial_t u)^2+(\px u)^2\big]\, dxdt
\ee
where $S=(0,T)\times (0,1)$. Since
\be\label{n3b}
\la\int_0^1(\partial_t u)(x\partial_x u)\, dx\ra
\le \mathcal{E}\defn
\mez \int_0^1 \big[(\partial_tu)^2+(\px u)^2\big]\,dx,
\ee
by using the conservation of energy ($d\mathcal{E}/dt=0$), we deduce
\be\label{n4}
\int_0^T(\px u(t,1))^2\, dt\ge (T-2)\int_0^1 \big[(\partial_tu)^2+(\px u)^2\big](0,x)\,dx.
\ee
This inequality implies that, for $T>2$, one can bound the energy by means of an observation at the boundary. 

There are inequalities analogous to \e{n4} 
which hold in multi-dimensional domains. 
In this direction, let us also mention the celebrated result 
by %Ralston \cite{Ralston} and 
Bardos--Lebeau--Rauch~\cite{BLR} which is 
devoted to the boundary observability of 
solutions of second-order hyperbolic equations in the general case when 
the region of observability meets every 
ray of geometric optics. 

The inequality \e{n4} is a central result in the study of the control 
of the wave equation---extensions and applications of \e{n4} are 
discussed in the first part 
of the SIAM Review article by Lions~\cite{Lions1988}; see also \cite{Coron,TW2009,MicuZuazua,Rosier1997}. 
In particular, (an extension of) \e{n4} allows to prove that in multi-dimensional domains, one can drive a solution of the wave equation from the rest position to a desired 
state, in finite time, by acting only on part of the boundary. 
Our intention in this paper is to start the study the analogous problem for the water-wave equations. 
This is a very natural question since it corresponds to the wavemaker problem: 
the aim is to determine which waves can be produced by, say, the motion 
of a plate immersed in a fluid. 

In this paper, our goal is to obtain a boundary observability result similar to \e{n4} for 
gravity waves. 
By contrast with the wave equation, 
the water-wave equations are fully nonlinear and contain nonlocal terms. 
However, we will prove that a similar observability inequality holds.

\begin{defi}\label{D11}
Introduce $\psi(t,x)=\phi(t,x,\eta(t,x))$ (evaluation of the 
potential at the free surface) and
$$
\Theta\defn - \eta\partial_t\psi-\frac{g}{2}\eta^2.
$$
Set 
$$
\mathscr{B}(T)=L_1L_2 \int_0^T
\left(\frac{1}{L_2}\int_0^{L_2} \Theta(t,L_1,x_2)\, dx_2+\frac{1}{L_1}\int_0^{L_1} \Theta(t,x_1,L_2)\, dx_1\right)\, dt.
$$
(It is convenient to use the factor $L_1L_2$ to derive some identities.)
\end{defi}
Notice that $\mathscr{B}(T)$ depends only on the boundary values of the unknowns. 
Our main result asserts that $\mathscr{B}(T)$ is larger than the energy $\mathcal{H}$ (cf \e{n2}) 
when $T$ is large enough.

\begin{theo}\label{TI}
Let $\beta>1/2$ and $\chi$ be either a $C^\infty_0(Q)$ function or $\chi\equiv 1$. 
There exist three positive constants $K_0,\kappa,c$ such that, for any $N$ in $\xN$, the following result holds. 
Assume that the initial data $\eta_0$ and $\psi_0$ are of the form 
\be\label{n302}
\chi(x)\sum_{|n|+|m|\le N}a_{nm}\cos\left(\pi \frac{n x_1}{L_1}\right)
\cos\left(\pi\frac{m x_2}{L_2}\right)\quad\text{with}\quad 
%\quad\text{and}\quad 
\la a_{nm}\ra\le cN^{-\kappa},
\ee
and such that $\eta_0$ satisfies \e{n9b}. 
Then  there exists a unique smooth solution $(\eta,\phi)$ of the water-wave 
equations \e{n5}--\e{n9b} such that $\eta(0,x)=\eta_0(x)$ and $\psi(0,x)=\psi_0(x)$. 
This solution exists on the time interval $[0,T_N]$ with 
$T_N=K_0+K_0N^\beta$ and satisfies
$$
\mathscr{B}(T_N)\ge \mathcal{H}.
$$
If, in addition, $\eta$ and $\phi$ are independent of $x_2$ (this means that we consider a two-dimensional wave), 
then one has the stronger conclusion 
$$
L_1\int_0^{T_N}\Theta(t,L_1,0)\, dt\ge\mathcal{H}.
$$
\end{theo}
We refer to Section~\ref{S:2} for an explanation of what we call a smooth 
solution of the water-wave equations. 
Several other remarks are in order. 

$(i)$ One can consider more general initial data, see Remark~\ref{R:final}.

$(ii)$ The second 
point to be made is a clarification of how one passes from a true observation at the boundary 
to a control of $\mathscr{B}(T)$ or $\int_0^{T_N}\Theta(t,L_1,0)\, dt$. Let us 
explain why 
these two quantities can be expressed as quantities depending only the restrictions to $\partial Q$ 
of $\eta$ and $\partial_t\eta$ 
(notice that $\eta,\partial_t\eta$ are quantities which can measured by a camera by contrast with $\psi$). 
To see this, consider firstly the case of a two-dimensional wave and 
set $m(t)=\eta(t,L_1,0)$. 
Then one has
$$
\Theta(t,M_1,0)=\mez \big[gm(t)^2-m(t)m'(t)^2\big],
$$
as can be verified by a simple calculation (see Remark~\ref{R:46}). For a three-dimensional wave, this is not so simple. 
However, 
one can determine $\psi$ from $(\eta,\partial_t\eta)$ by using the equation \e{n8} and considering the 
Neumann to Dirichlet operator (as in Appendix A.3 in Lannes' book \cite{LannesLivre}).
%; proceeding as explained in Section \ref{S:2} to take into account the boundary condition \e{n6}--\e{n6b}). 

$(iii)$ Theorem~\ref{TI} will be deduced from an observability estimate which holds without smallness assumptions. 
Namely, we will prove that for all smooth enough solution defined on the time interval $[0,T]$, one has (cf \e{n501})
$$
\mathscr{B}(T)\ge \left(\frac{T}{2}-\frac{5+2d}{4} B T\right)\mathcal{H}
-\left(d+\tdm \right)\max\{L_1,L_2\}\frac{\sqrt{2\mathcal{H}}}{\sqrt{g}}A
$$
where
$$
B\defn \sup_{t\in [0,T]}\lA \nabla\eta(t)\rA_{L^\infty},\quad 
A\defn \sup_{t\in[0,T]}\lA \nabla\psi(t)\rA_{L^2}.
$$
The assumption \e{n302} will be used only to prove that one can bound $A$ in terms of $\sqrt{\mathcal{H}}$ on large time intervals. 
Then, when $T$ is large enough, the right-hand side of the above inequality is larger than $\mathcal{H}$.

$(iv)$ Theorem \ref{TI} gives an observability inequality 
in time $N^{\beta}$ for $\beta>1/2$. The important think to note is that, 
except for the fact that we preclude $\beta=1/2$, 
this result is sharp. Indeed, an harmonic with frequency $N$ travels at a speed proportional to $1/\sqrt{N}$ 
(the dispersion relationship for the linearized equations is $\omega(k)^2=g|k|$). 
So it might take a time $\sqrt{N}$ to reach the boundary and hence for the observation to be possible. 
Now we can explain why we need a smallness assumption on the initial data, 
namely the assumption $|a_{nm}|\le cN^{-\kappa}$. This 
is because one needs such a smallness assumption to guarantee that the solutions exist 
on large time intervals of size $N^\beta$ (the link between $\kappa$ and $\beta$ 
will be made through an interpolation argument in the proof). 

$(v)$ Another point should be added: as indicated on Figure~\ref{fig1} below, 
we are making an observation on part of the boundary only (satisfying the so-called 
geometric control condition). Indeed, 
to determine $\int_0^{T_N}\Theta(t,L_1,0)\, dt$ 
it is sufficient to look at the motion of the point of contact $\mathscr{P}(t)=(L_1,0,\eta(t,L_1,0))$ 
between the free surface and only one wall, namely $\{x_1=L_1\}$. 
Similarly, for a three-dimensional wave, 
to determine $\mathscr{B}(T)$, it is sufficient to observe 
the motion of the curves of contact between the free surface and the walls $\{x_1=L_1\}$ and $\{x_2=L_2\}$:
\begin{align*}
&\mathscr{C}_1(t)=\{(L_1,x_2,y),~x_2\in [0,L_2],~y=\eta(t,L_1,x_2)\},\\
&\mathscr{C}_2(t)=\{(x_1,L_2,y),~x_1\in [0,L_1],~y=\eta(t,x_1,L_2)\}.
\end{align*}

\begin{figure}
\begin{minipage}[b]{0.45\linewidth}
\begin{tikzpicture}[samples=100,scale=0.6]
\tikzstyle{ann} = [fill=white,font=\footnotesize,inner sep=1pt]
%\draw [<-,gray] (-2.2,-0.5) -- (-2.2,-2) ;
%\draw [->,gray] (-2.2,-0.5) -- (-2.2,-2) ; 
%\draw [<-,gray] (-2,-0.3) -- (2,3.2) ;
%\draw [->,gray] (-2,-0.3) -- (2,3.2) ;
%\node at (-2.2,-1) {$\eps$};
%\draw [->,gray] (-2.2,-0.5) -- (-2.2,-2) ; 
\draw [gray] (-2,0) -- (-2,-4.5) ;
\draw [gray] (-2,0) -- (2,3.5) ;
\draw [gray] (2,3.5) -- (2,-1) ;
\draw [color=black!70!white,line width=0.5pt,->] (2,3.5) -- (2,4.3) node [right] {$y$};
\draw [gray] (2,3.5) -- (8,3.5) ;
\draw [arrows=<->,dashed,gray] (2.05,3.9) -- (8,3.9) ;
\node [gray,ann] at (4.9,4) {$L_1$};
\draw [arrows=<->,dashed,gray] (-2,0.4) -- (1.95,3.9) ;
\node [gray,ann,above] at (0.1,2.2) {$L_2$};
\draw [gray] (8,3.5) node [right,below] {$\qquad\mathscr{R}$} -- (8,-0.5) ;
\draw [color=black!70!white,line width=0.5pt,->] (8,-0.5)  -- (8.7,-0.5) node [below] {$x_1$} ;
\draw [gray] (-2,-4.5) -- (4,-4.5) ;
 \shadedraw [top color=blue!60!white, bottom color=blue!30!white] %fill=blue!20!white, draw=gray] 
 (-2,-2) -- (-2,-4.5) -- (4,-4.5) -- (4,-2) ;
  \shadedraw [right color=blue!60!white, left color=blue!50!white, color=gray] %fill=blue!20!white, draw=gray] 
 (-2,-2) -- (-2,-4.5) -- (4,-4.5) -- (4,-2) ;
 \shadedraw  [left color=blue!60!white, right color=blue!30!white, color=gray] 
(4,-4.5) -- (4,-3) -- 
 plot [domain=-2:2,shift={(6,-1)},rotate=0.3] ({\x},{\x+exp(-\x*\x)}) -- (8,-0.5) ;
\shadedraw [top color=blue!30!white, bottom color=blue!60!white, color=gray] plot [domain=-2:2,rotate=0.3] ({\x},{\x+exp(-\x*\x)}) 
-- plot [domain=-2:4,shift={(4,1.55)}] ({\x},{(0.5*cos((0.524*(\x+2) r)))}) -- 
 plot [domain=2:-2,shift={(6,-1)},rotate=0.3] ({\x},{\x+exp(-\x*\x)}) -- plot [domain=4:-2,shift={(0,-2.5)}] ({\x},{(0.5*cos((0.524*(\x+2) r)))});
% \draw [gray] (-2,-2) -- (4,-2) ;
% \draw [gray] (0.1,1.1) -- (6.1,1.1) ;
\draw[color=black!20!white,domain=2:-2,shift={(1,-0.06)},rotate=0.3] plot  ({\x},{\x+exp(-\x*\x)}) ;
\draw[color=black!20!white,domain=2:-2,shift={(2,-0.25)},rotate=0.3] plot  ({\x},{\x+exp(-\x*\x)}) ;
\draw[color=black!20!white,domain=2:-2,shift={(3,-0.5)},rotate=0.3] plot  ({\x},{\x+exp(-\x*\x)}) ;
\draw[color=black!20!white,domain=2:-2,shift={(4,-0.75)},rotate=0.3] plot  ({\x},{\x+exp(-\x*\x)}) ;
\draw[color=black!20!white,domain=2:-2,shift={(5,-0.93)},rotate=0.3] plot  ({\x},{\x+exp(-\x*\x)}) ;
\draw[color=black!20!white] (-1,-2.06) -- (-1,-4.5);
\draw[color=black!20!white] (0,-2.25) -- (0,-2.5);
\draw[color=black!20!white] (0,-3.3) -- (0,-4.5);
\draw[color=black!20!white] (1,-2.5) -- (1,-2.6);
\draw[color=black!20!white] (1,-3.3) -- (1,-4.5);
\draw[color=black!20!white] (2,-2.75) -- (2,-4.5);
\draw[color=black!20!white] (3,-2.93) -- (3,-4.5);
%\draw [color=blue!20!white] (-2,-2) -- (4,-2);
\draw [color=blue!60!white] (-2,-4.5) -- (4,-4.5);
\node at (0,-4) {$~$};
%\draw (-2,-2) -- (4,-2) ;
%\draw plot ({\x},{0.8*cos((\x-2) r)}); 
%\draw[color=black,domain=4:-2,shift={(4,3.5)}] plot ({\x},{(0.5*cos((0.524*(\x+2) r)))})
\draw [gray] (4,0) -- (4,-4.5) ;
\draw [gray] (4,0) -- (8,3.5) ;
\draw [gray] (4,0) -- (-2,0) ;
\draw [gray] (4,-4.5) -- (8,-0.5) ;
\draw [color=black!70!white,line width=0.5pt,->] (-2,-4.5) -- (-2.5,-5) ;
\node at (-2.5,-5.2) [color=black!70!white,right] {$x_2$};
\draw[style=thick,color=black,domain=4:-2,shift={(0,-2-0.5)}] plot  ({\x},{(0.5*cos((0.524*(\x+2) r)))}) ;
\node at (0.6,-2.4) [below] {$\mathscr{C}_2(t)$};
\draw[style=thick,color=black,domain=2:-2.018,shift={(6,-1)},rotate=0.3] plot  ({\x},{\x+exp(-\x*\x)}) ;
\node at (6.5,0) [below] {$\mathscr{C}_1(t)$};
%\draw [->,style=thick,color=black!80!white] (4.3,1.5) -- (4.45,2.3) node [right] {$n$};
%\draw [->,style=thick,color=black!80!white] (4.7,-3.5) -- (5.6,-3.5) node [right] {$\nu$};
\end{tikzpicture}
\end{minipage}
\hspace{5mm}
\begin{minipage}[b]{0.45\linewidth}
%\begin{figure}[!ht]
%\centerline{
\begin{tikzpicture}[scale=1.1,samples=100]
\shadedraw [top color=blue!70!white, bottom color=blue!30!white] 
%\filldraw[fill=blue!20!white, draw=gray] 
(-3,0) -- (-3,-3) -- (1,-3) -- (1,0)  -- 
(1,-0.7) to [out=180,in=50]
(-2,-0.5) to [out=230,in=0] (-3,-1.2) -- (-3,-1) ;
%\draw [gray] (-3,-1) -- (-3,-0.5) ;
%\draw [gray] (1,-1) -- (1,-0.5) ;
\node at (1,-0.7) {\textbullet};
\node at (1,-0.7) [right] {$\mathscr{P}(t)$};%{$(1,\eta(t,1))$}; 
\node at (-1,-1.7) {$\Omega(t)$};
\draw [color=black!70!white,->,dashed] (-3.5,-3.5) -- (-3.5,0.5) node [right] {$y$};
\draw [color=black!70!white,->,dashed] (-3.5,-3.5) -- (1.5,-3.5) node [below] {$x_1$};
\draw (-3,-3.4) -- (-3,-3.6)  node[color=black!70!white,below] {$0$};
\draw (1,-3.4) -- (1,-3.6)  node[color=black!70!white,below] {$L_1$};
\draw (-3.6,-0.6) -- (-3.4,-0.6)  ;
\node at (-3.6,-0.6) [color=black!70!white,left] {$0$};
\draw (-3.6,-3) -- (-3.4,-3)  ;
\node at (-3.6,-3) [color=black!70!white,left] {$-h$};
\end{tikzpicture}
%}
%\end{figure}
\end{minipage}
\caption{Three-dimensional and two-dimensional waves  
in a rectangular tank\label{fig1}}
\end{figure}
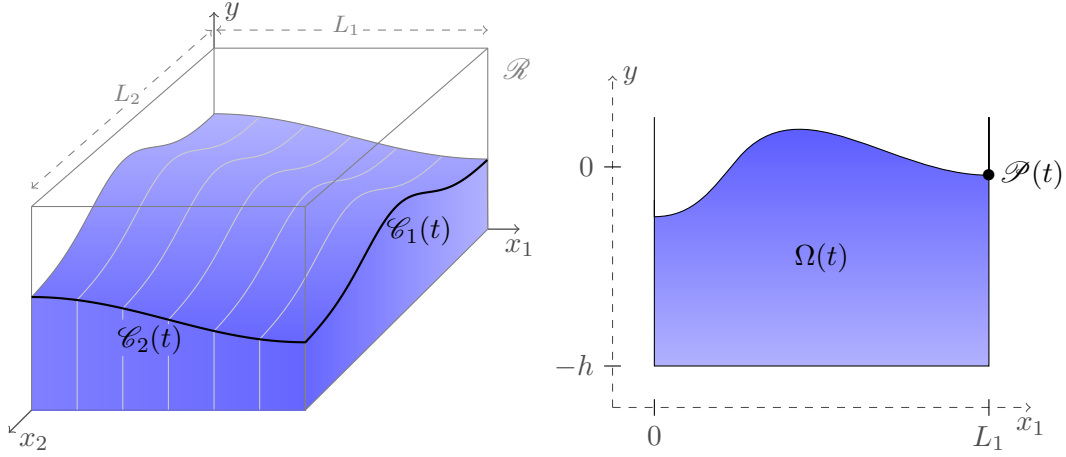

\subsection{Strategy of the proof and main identity}

In the rest of this introduction, 
for the sake of simplicity, we consider two-dimensional waves (these are solutions independent of $x_2$, and 
we write simply $\eta(t,x_1)$ and $\phi(t,x_1,y)$). We refer to 
the next sections for the corresponding statements for three-dimensional waves.

Theorem~\ref{TI} will be deduced from  
an exact identity, similar to \e{n3}, where, quite surprisingly, 
the terms coming from the nonlinear part of the equations have a very simple and compact form.

\begin{theo}\label{T1}
Consider a smooth enough solution of the water waves 
equations \e{n5}--\e{n9b} defined on the time interval $[0,T]$ and independent of $x_2$. 
Introduce 
$$
m(t)=\eta(t,L_1).
$$
Then 
$\Theta(t,L_1)=\mez \big[gm(t)^2-m(t)m'(t)^2\big]$ and 
\be\label{n10}
\begin{aligned}
L_1\int_0^T \Theta(t,L_1)\, dt&=\frac{T}{2}\mathcal{H}\\
&\quad+\frac{L_1}{2}\int_0^T\int_{-h}^{m(t)}(\py\phi)^2(t,L_1,y)\, dydt \\[1ex]
&\quad+\mez\int_0^T\int_0^{L_1} \left(h+\frac{7}{4}\eta\right)(\px\phi)^2(t,x,-h)\, dxdt\\[1ex]
&\quad-\uq \int_0^{L_1} \eta \psi\,dx\Big\arrowvert_{t=0}^{t=T}
-\int_0^{L_1} x\eta\px\psi\, dx\Big\arrowvert_{t=0}^{t=T}\\[1ex]
&\quad-\frac{7}{4}\int_0^T\iint_{\Omega(t)} (\px\eta)(\px\phi)(\py\phi)\, dxdydt,
\end{aligned}
\ee
where $\mathcal{H}$ is given by \e{n2} and $\int f\,dx\arrowvert_{t=0}^{t=T}$ 
stands for $\int f(T,x)\,dx-\int f(0,x)\,dx$.
\end{theo}

Theorem~\ref{T1}Ê
is proved in Section~\ref{S:4}. 
The proof uses the Zakharov's formulation of the water-wave problem 
as a Hamiltonian system (see~\cite{Zakharov1968})  and 
the observation by Craig and Sulem \cite{CrSu} 
that the equations and the hamiltonian are most naturally expressed in terms of the 
Dirichlet to Neumann operator $G(\eta)$. The main ingredients of the proof of Theorem~\ref{T1} are then: 
$i)$ a Pohozaev identity for the Dirichlet to Neumann operator (that is a computation of $\int (G(\eta)\psi)x\px\psi \, dx$) 
which shows that the contributions due to the boundary 
conditions are positive and $ii)$ some computations inspired by the analysis 
of Benjamin and Olver \cite{BO} of the conservation laws for water waves. 
Let us mention that, in the appendix, we give another proof of \e{n10} 
which exploits the hamiltonian structure of the water-wave equations. The idea is 
to compute
$$
A\defn \iint_{[0,T]\times [0,L_1]} \big[(\partial_t\eta)(x\px\psi)-(\partial_t\psi)(x\px\eta)\big]\, dxdt.
$$
We compute $A$ in two different ways, and the wanted identity \e{n10}Ê
is obtained by comparing the two results. The first computation is simply an integration by parts. 
The second computation relies on the fact that the equation is hamiltonian (see Zakharov~\cite{Zakharov1968}, Craig--Sulem~\cite{CrSu}), which means that
$$
\partial_t\eta=\frac{\delta \mathcal{H}}{\delta \psi},\quad \partial_t\psi=-\frac{\delta \mathcal{H}}{\delta \eta}.
$$
Therefore
$$
A=\int_0^T a(t)\, dt\quad \text{with}\quad a\defn \int_0^{L_1} \left[x\psi_x\frac{\delta \mathcal{H}}{\delta \psi}
+x\eta_x\frac{\delta \mathcal{H}}{\delta \eta}\right]\, dx,
$$
Then we compute $a$ by writing
$$
a=\lim_{\eps\rightarrow 0}\frac{1}{\eps}\left[ \mathcal{H}(\eta,\psi+\eps x\psi_x)-\mathcal{H}(\eta,\psi)\right]
+\lim_{\eps\rightarrow 0}\frac{1}{\eps}\left[ \mathcal{H}(\eta+\eps x\eta_x,\psi)-\mathcal{H}(\eta,\psi)\right],
$$
and using a shape derivative formula due to Lannes \cite{LannesJAMS}. This allows to 
avoid the use of the Pohozaev identity. However, the proof still requires to prove some identities for quantities 
which are analyzed as conservation laws.

We now indicate how 
to infer from Theorem~\ref{T1} a boundary observability result. 
This will require to make an additional assumption. Indeed, 
in sharp contrast with the example of the wave equation \e{n3a} discussed above, 
even for the linearized water-wave equations at the origin, 
there is no observability inequality in finite time which applies to 
all initial data (see Reid and Russell~\cite{ReidRussell1985} or Biccari~\cite{Biccari} for related results). 
However, we will prove a boundary observability result, 
for the full nonlinear problem, for initial data whose Fourier transforms are compactly supported. 
To do so, we begin with the following corollary of Theorem~\ref{T1}. 

\begin{coro}\label{C2}
Consider a smooth enough solution of the water waves 
equations \e{n5}--\e{n9b} defined on the time interval $[0,T]$. % and such that $\eta\ge -\min\{h/2,1/2\}$. 
Assume that there exists two positive constants $A,B$ such that
\be\label{n102}
B<\frac{2}{7},\quad T\ge \frac{4}{2-7B}\left(1+\frac{5L_1}{\sqrt{g}}A\right),
\ee
and
\be\label{n105}
\sup_{t\in [0,T]}\lA \partial_x\eta(t)\rA_{L^\infty([0,L_1])}\le B,\quad 
\sup_{t\in[0,T]}\lA \px\psi(t)\rA_{L^2([0,L_1])}\le A\sqrt{2\mathcal{H}}.
\ee
Then
\be\label{n100i}
\begin{aligned}
L_1\int_0^T\Theta(t,L_1)\, dt &\ge \mathcal{H}\\
&\quad+\frac{L_1}{2}\int_0^T\int_{-h}^{m(t)} (\py\phi)^2(t,L_1,y)\, dydt \\
&\quad +\frac{h}{16} \int_0^T\int_0^{L_1} (\px\phi)^2(t,x,-h)\, dxdt.
\end{aligned}
\ee
\end{coro}\begin{proof}
The proof is straightforward: directly from \e{n105}, the definition of $\mathcal{H}$, 
the Cauchy-Schwarz inequality 
and the assumption $\int_0^{L_1}\eta\, dx=0$, we get
\begin{align*}
&\la \int_0^{L_1} \eta \psi\,dx\ra
=\la \int_0^{L_1} \eta \left(\psi-\int_0^{L_1}\psi\, dx\right)\,dx\ra \le L_1\lA \eta\rA_{L^2_x}\lA \px\psi\rA_{L^2_x}\le 
(2L_1/\sqrt{g})A\mathcal{H} ,\\
&\la \frac{7}{4}\int_0^T\iint_{\Omega(t)} (\px\eta)(\px\phi)(\py\phi)\, dxdydt\ra 
\le  \frac{7B}{4}T\mathcal{H},\\
&\la \int_0^{L_1} x\eta\px\psi\, dx\ra\le L_1\lA \eta\rA_{L^2_x}\lA \px\psi\rA_{L^2_x}\le (2L_1/\sqrt{g})A\mathcal{H},
\end{align*}
where we used the Poincar\'e inequality in the first line. 
Now, since $\eta\ge -h/2$ by assumption, 
using the identity \e{n10} we see that the 
left hand side of \e{n100i} is larger than the sum of the last two terms in the right hand side of \e{n100i} and of
$$
\left(\mez T-\frac{5L_1}{\sqrt{g}}A-\frac{7B}{4}T\right)\mathcal{H},
$$
which, by assumption \e{n102}, is larger than $\mathcal{H}$.
\end{proof}

\begin{rema}\label{Rhami}
$i)$ One controls not only the energy $\mathcal{H}$ 
but also the traces of the velocities on the wall $\{x=L_1\}$ and on the bottom. 

$ii)$ ({\em Unique continuation}) If $m=0$ then \e{n100i}Ê
implies that 
$\mathcal{H}=0$ and hence $\eta=0$ and $\phi=0$. 

$iii)$ The assumption $\la\px\eta(t,x)\ra\le 2/7$  
is physically realistic. 
Indeed, one expects a steepness-induced blow-up in finite time when the  
wave height is large compared to the wavelength (there are no mathematical proof of this claim 
but it is known that blow-up occurs in finite time for some large enough initial data, see \cite{CCFGGS}). 
Moreover, 
the threshold $2/7$ is in good agreement with the sharp experimental studies of steep irregular wave events 
reported in \cite{JCHG}. 

\end{rema}

To conclude this introduction, we explain how we deduce Theorem~\ref{TI}Ê
from 
Corollary~\ref{C2}. The reason why we need the assumption for $\psi$ in \e{n105} is the following: 
we do not have an estimate analogous to 
\e{n3b} for the term $\int_0^1 x\eta\px\psi\, dx$ which 
appears in the right-hand side of \e{n10}. 
Indeed, the fact that $\sqrt{2\mathcal{H}}$ is larger 
than the $L^2(\Omega(t))$-norm of $\nabla_{x,y}\phi(t)$ 
gives only a bound of the $\dot H^{1/2}_x$-norm of the trace $\psi$. 
Hence $\sqrt{\mathcal{H}}$ 
does not control the $L^2$-norm of $\px\psi$ and this why we need the assumption for $\psi$ in \e{n105}. 
Now, notice that this assumption for $\px\psi$ 
holds at $t=0$ with $A=K\sqrt{N}$ if the Fourier transform of $\psi(0)$ is supported 
in $[-N,N]$ (as in \e{n302}). Since the equations are nonlinear, we cannot assume 
that the Fourier transform of the solution will be supported in $[-N,N]$ for all time. 
However, we shall see that, for small data, one can propagate 
the estimate $\lA \px\psi(0)\rA_{L^2}\les K\sqrt{N}\lA \psi(0)\rA_{\dot H^{1/2}}$ on large time intervals and hence deduce and observability result. To do so, we shall combine 
an interpolation argument and 
the fact that the Cauchy problem is well-posed on $H^s$ (with $s$ large enough) on large time intervals for small initial data.

\subsection*{Organization of the paper.} In Section~\ref{S:2} 
we recall various results about the well-posedness of the Cauchy problem. 
This will allow us to clarify what we call a smooth enough solution. 
In Section~\ref{S:Pohozaev}, we prove a Pohozaev identity for the 
Dirichlet to Neumann operator and we use this identity in Section~\ref{S:4} to prove Theorem~\ref{T1}. 
Then we prove Theorem~\ref{TI} in Section~\ref{S:5}. 
In the appendix we give a proof of Theorem~\ref{T1} 
which exploits the hamiltonian structure of the equations.

\subsection*{Acknowledgements}

%Many thanks to Emmanuel Dormy and Christophe Lacave for 
%letting me reproduce the outcome of their numerical 
%simulations. 
%Also, 
I warmly thank Nicolas Burq for stimulating discussions about 
the control and the observability of partial differential equations.

\section{About the Cauchy problem in a rectangular tank}\label{S:2}

We recall here various results about the well-posedness of the Cauchy problem. 

\subsection{The Craig--Sulem--Zakharov system}

To study the Cauchy problem for the water-wave equations, there are at least two difficulties. 
Firstly, one has to reduce the analysis to a time independent domain. 
In the case where the space variable $x$ belongs to the torus instead of $Q$, several approaches have been used. 
In this paper, we use the Eulerian approach, following Zakharov \cite{Zakharov1968}, Craig-Sulem \cite{CrSu}, 
Lannes~\cite{LannesJAMS,LannesLivre} and Alazard--Burq--Zuily~\cite{ABZ3,ABZ4}. In this approach, one works 
with the trace of the potential $\phi$ at the free surface and the Dirichlet to Neumann operator $G(\eta)$. 

Hereafter we consider either the 2D problem or the 3D problem (see figure~\ref{fig1}) and denote by $d\in \{1,2\}$ 
the dimension of the free surface. 
For a three-dimensional fluid one has $d=2$ and we use the notation 
$\nabla=(\partial_{x_1},\partial_{x_2})$ (and we also use the notation $\nabla=\partial_x$ when $d=1$). 

We set
$$
\psi(t,x)=\phi(t,x,\eta(t,x)),
$$
and introduce the Dirichlet to Neumann operator $G(\eta)$ defined by
$$
G(\eta)\psi=\sqrt{1+|\nabla\eta|^2}\, \partial_n\phi\arrowvert_{y=\eta}=(\partial_y \phi -\nabla\eta\cdot \nabla\phi)\arrowvert_{y=\eta}
$$
(see the next paragraph for a precise definition of $G(\eta)$).  
Then the Craig--Sulem--Zakharov formulation of the water-wave equations reads
\begin{equation}\label{system}
\left\{
\begin{aligned}
&\partial_t \eta=G(\eta)\psi,\\
&\partial_t \psi + g\eta+ \frac{1}{2}\la \nabla\psi\ra^2
-\frac{1}{2(1+\la \nabla\eta\ra^2)}\bigl(G(\eta)\psi+\nabla\eta\cdot \nabla\psi\bigr)^2= 0.
\end{aligned}
\right.
\end{equation}
Moreover, the energy $\mathcal{H}$ satisfies 
$$
\mathcal{H}=\frac{g}{2}\int\eta^2\, dx +\mez\iint \la \nabla_{x,y}\phi\ra^2\, dydx
=\mez \int \big[g\eta^2+\psi G(\eta)\psi\big]\, dx.
$$
We refer the reader to the original article by Craig--Sulem \cite{CrSu} 
for the derivation of this system from the water-wave equations 
\e{n5}--\e{n9} as well as to \cite{Bertinoro}Ê
for the proof that, conversely, given a solution of \e{system},Ê
one can define a solution 
of \e{n5}--\e{n9}.

\subsection{Definition of the Dirichlet to Neumann operator}\label{S:22}

We have already written the water waves equations under the form of the Craig-Sulem-Zakharov system \e{system}. 
We give here the precise definition of the Dirichlet to Neumann 
operator that is used in that system in the context of periodic functions.

For $s\in [0,+\infty)$, we denote by 
$H^s(\xT^d)$ 
the Sobolev space of periodic functions:
$$
H^s(\xT^d)=\Big\{ u=\sum_{n\in\xZ^d}a_n e_n(x)\, ;\,\sum_{n\in\xZ^d}(1+|n|^2)^{s}\la a_n\ra^2<+\infty\Big\}.
$$
where ${\displaystyle e_n(x)=\exp\Big(i\pi \frac{nx}{L_1}\Big)}$ if $d=1$, 
${\displaystyle e_n(x)=\exp\Big(i\pi \Big(\frac{n_1x_1}{L_1}+\frac{n_2x_2}{L_2}\Big)\Big)}$ when $d=2$ 
(we fix the periods equal to $2L_1,2L_2$ for reasons that will be clear below).

Fix $h>0$ and consider $\eta\in H^{s_0}(\xT^d)$ with $s_0>d/2+1$ such that 
$\eta(x)\ge -h/2$. Set
\be\label{n231}
\widetilde{\Omega}= \{(x,y)\in \xR^d\times\xR\, :\, -h\le y\le\eta(x)\}.
\ee
Since $\eta$ is Lipschitz, by the usual variational method, one obtains that, for any 
$\psi\in H^{1/2}(\xT^d)$, the problem 
\be\label{n170}
\Delta_{x,y}\phi=0\quad\text{in }\widetilde{\Omega}, \quad \phi\arrowvert_{y=\eta}=\psi, 
\quad \partial_y\phi\arrowvert_{y=-h}=0,
\ee
has a unique variational solution. Moreover, even if $\nabla_{x,y}\phi$ belongs only to $L^2(\widetilde{\Omega})$, one can prove that
$$
G(\eta)\psi=(\partial_y \phi -\nabla\eta\cdot\nabla \phi)\arrowvert_{y=\eta},
$$
is well-defined and belongs to $H^{-1/2}(\xT^d)$. More generally, one has the following result (see \cite{ABZ3,LannesLivre}).

\begin{prop}\label{P7}
Let $s_0>d/2+1$ and assume that $\eta$ is in $H^{s_0}(\xT^d)$. Then 
$G(\eta)$ is a bounded operator from $H^{\sigma}(\xT^d)$ into $H^{\sigma-1}(\xT^d)$ for any $\sigma\in [1/2,s_0]$. 
\end{prop}

Later we shall perform various integrations by parts and apply repeatedly the Green's identity. 
The fact that all the computations are meaningful relies on the following regularity result.
\begin{prop}\label{P6}
Let $d\ge 1$, $s>d/2+2$ and 
assume that $(\eta,\psi)\in H^{s}(\xT^d)\times H^s(\xT^d)$. Then 
$$
\nabla_{x,y}\phi\in C^1(\overline{\widetilde{\Omega}}).
$$
\end{prop}
\begin{proof}We explain how to deduce this result from the results in Lannes's book \cite{LannesLivre}. 

It is convenient to flatten the free surface. To do so, consider a 
diffeomorphism from $\xR^d\times [-h,0]$ to $\widetilde{\Omega}$,  
of the form 
$(x,z) \mapsto (x, \rho(x,z))$. The simplest choice would be to set
\begin{equation}\label{diffeo}
\rho(x,z)=  \left(1+\frac{z}{h}\right)\eta(x)+z.
\end{equation}
However, since we only assume that $s>d/2+2$ (the analysis is simpler for $s>d/2+5/2$), 
following Lannes~\cite[\S 2.2.2]{LannesLivre}, 
we need to consider a regularized version of \e{diffeo}. Introduce
$$
\zeta(x,z)=c\sum_{n\in\xZ^d}e^{-\delta z^2 |n|^2}\hat\eta_n e_n(x),\qquad 
\hat\eta_n=\int_{\xT^d}\overline{e_n(y)}\eta(y)\, dy,
$$
where $c$ and $\delta$ are positive constants (with $\delta$ small enough), chosen so that
\be\label{n22x}
\zeta\arrowvert_{z=0}=\eta, \quad 
\la \partial_z \zeta(x,z)\ra\le \min \Big\{ \frac{h}{4},\frac{1}{20}\Big\},\quad 
\zeta\in H^{s+\mez}(\xT^d\times [-h,0]).
\ee
Then set
\be\label{reg}
\rho(x,z)=\left(1+\frac{z}{h}\right)\zeta(x,z)+z.
\ee
Notice that $\rho(x,0)=\eta(x)$ and $\rho(x,-h)=-h$. 
Recall that $\eta\ge -h/2$ by assumption. In view of the bound for $\pz\zeta$ in \e{n22x}, we deduce that 
$\zeta\ge-3h/4$ and $\partial_z \rho\ge 1/5$, which proves that $(x,z) \mapsto (x, \rho(x,z))$ is a $C^2$-diffeomorphism. 
The problem thus reduces to establishing that the function $\varphi(x,z)=\phi(x,\rho(x,z))$ 
satisfies $\nabla_{x,z}\va\in C^1(\xT^d\times [-h,0])$. To obtain this result, 
we view $\va$ as a function of $z$ with values in functional spaces. 
Since $s>d/2+2$ and since we chose a special change of variables where $\rho$ is given by \e{reg}, we are in position to apply 
Corollary $2.40$ in Lannes' book \cite{LannesLivre} (with $s$ replaced by $s-1/2$). We 
deduce that
$$
\nabla_{x,z}\va\in L^2_z([-h,0];H^{s-\mez}(\xT^d))\cap H^1_z([-h,0];H^{s-\tdm}(\xT^d)).
$$
As a result
\begin{align*}
&\nabla_{x}\va\in L^2_z([-h,0];H^{s-\mez}(\xT^d))\quad\text{and}\quad \partial_z \nabla_x \va 
\in L^2_z([-h,0];H^{s-\tdm}(\xT^d)),\\
&\partial_z\va\in L^2_z([-h,0];H^{s-\mez}(\xT^d))\quad\text{and}\quad \partial_z \partial_z \va 
\in L^2_z([-h,0];H^{s-\tdm}(\xT^d)),
\end{align*}
and hence
\be\label{reg6}
\nabla_{x,z}\va\in C^0_z([-h,0];H^{s-1}(\xT^d)).
\ee
Then, using the second order equation satisfied by $\va$ one can express $\pz^2\va$ in terms of $(\pz\va,\pz\nabla_x\va,\nabla_x^2\va)$ 
and in terms of $\rho$ and hence, using the standard nonlinear estimates in Sobolev spaces, one deduces that
\be\label{reg7}
\pz^2\va\in C^0_z([-h,0];H^{s-2}(\xT^d)).
\ee
The wanted result $\nabla_{x,z}\va\in C^1(\xT^d\times [-h,0])$ then follows from \e{reg6}--\e{reg7} 
and the Sobolev embeddings $H^{s-1}(\xT^d)\subset C^1(\xT^d)$ and $H^{s-2}(\xT^d)\subset C^0(\xT^d)$. 
\end{proof}

\subsection{The Cauchy problem for periodic functions}

We recall here a well-posedness result for the Cauchy problem for periodic functions. 
There are now quite a lot of papers on this subject 
and we quote below only two results which are related 
to our problem. The main difficulty in the analysis of System \e{system} 
is that writing energy estimates on the function $(\eta,\psi)$ makes appear an apparent
loss of half a derivative. A way to circumvent that difficulty is to bound the energy not 
of $(\eta,\psi)$ but $(\eta,B,V)$ where
\begin{equation}\label{n10b}
B=\py\phi\arrowvert_{y=\eta}=\frac{G(\eta)\psi+\nabla \eta\cdot\nabla\psi}{1+|\nabla\eta|^2},\quad
V=\nabla\phi\arrowvert_{y=\eta}=\nabla\psi-B\nabla\eta.
\end{equation}
If initially $\eta_0$ and $\psi_0$ belong to $H^{s_0+\mez}(\xT^d)$ for some $s_0>(d+1)/2$, then 
$G(\eta_0)\psi_0$ belongs to $H^{s_0-\mez}(\xT^d)$ (cf Proposition~\ref{P7}). 
On the other hand, $H^{s_0-\mez}(\xT^d)$ is an algebra for any $s_0>(d+1)/2$. 
It thus follows from usual nonlinear estimates in Sobolev spaces that 
\begin{equation}\label{n98}
B_0=\frac{G(\eta_0)\psi_0+\nabla\eta_0 \cdot\nabla\psi_0}{1+|\nabla\eta_0|^2}\in H^{s_0-\mez}(\xT^d),\quad
V_0=\nabla\psi_0-B_0\nabla\eta_0\in H^{s_0-\mez}(\xT^d).
\end{equation}
The following result shows that one can propagate the fact that 
$B_0$ and $V_0$ are in $H^{s_0}(\xT^d)$ for $s_0>d/2+1$ (and this is the key point to circumvent the 
apparent loss of half of 
derivative in the study of the Cauchy problem).

\begin{theo}[from Alazard-Burq-Zuily \cite{ABZ3}]\label{T2a}
Let $d\ge 1$, $s_0>d/2+1$ and consider an initial data $(\eta_0,\psi_0)$ such that
\be\label{n10a}
(\eta_0,\psi_0,V_0,B_0)\in X^{s_0}\defn 
H^{s_0+\mez}(\xT^d)\times H^{s_0+\mez}(\xT^d)\times H^{s_0}(\xT^d)\times H^{s_0}(\xT^d).
\ee
Then there exists a time $T>0$ such that the Cauchy problem 
for \e{system} with initial data $(\eta_0,\psi_0)$ 
has a unique solution 
$(\eta,\psi)$ such that $(\eta,\psi,V,B)\in C^{0}([0,T];X^{s_0})$.
\end{theo}

\subsection{Extension to periodic functions}\label{S:canal}

We recall here from Alazard-Burq-Zuily \cite{ABZ4} how to solve 
the Cauchy problem for the water-wave equations in a rectangular tank. 
Let $d\in \{1,2\}$, $Q=[0,L_1]$ if $d=1$ and $Q=[0,L_1]\times [0,L_2]$ if $d=2$. 
One denotes by $\nu$ the outward unit normal to $Q$ ($\nu=(1,0)$ if $x_1=L_1$, $\nu=(0,-1)$ if $x_2=0$,...). 

As recalled in the introduction, 
the key observation is the following: for smooth enough solutions, 
the angle between the free surface and the vertical boundary of the tank is a right angle (see Section~$6$ in \cite{ABZ4}). 
This means that $\partial_\nu \eta=0$ on $\partial Q$. 
Now observe that $\nabla\psi=(\nabla\phi)\arrowvert_{y=\eta} +(\partial_{y}\phi )\arrowvert_{y=\eta}\nabla\eta$. 
Since 
$\partial_{n}\phi=0$ on the boundary of the tank $\mathscr{R}=Q\times [-h,+\infty)$, we conclude that 
$\partial_\nu \psi=0$ on $\partial Q$.
\begin{defi}
Given $1\le d\le 2$ and $\sigma>3/2$, one denotes by $H^\sigma_e(Q)$ the space
$$
H^\sigma_e(Q)=\{v\in H^\sigma(Q): \partial_\nu v=0\text{ on }\partial Q\}.
$$
\end{defi}

Consider the Cauchy problem for initial data $\eta_0,\psi_0\colon Q \rightarrow \xR$ in $H^\sigma_e(Q)$ for some $\sigma$ large enough.
Following Boussinesq (see~\cite[page 37]{Boussinesq}), the idea is that, in a general setting, one can extend these initial data 
to periodic functions defined for $x\in\xR^d$, solve the Cauchy problem for these extended initial data 
and then obtain a solution to the water-wave equations in a canal by considering the restrictions of these solutions.

\begin{defi}\label{De:tilde}
Let $v\colon Q\rightarrow\xR$. If $d=2$, 
we define $\widetilde{v}\colon \xR^2\rightarrow \xR$ as the unique extension of $v$ satisfying
\begin{align}
&\widetilde{v}(x)=v(x)\quad \forall x\in Q,\\
&\widetilde{v}(-x_1,x_2)=\widetilde{v}(x_1,x_2)=\widetilde{v}(x_1,-x_2) \quad \forall x\in \xR^2,\label{n210b}\\
&\widetilde{v}(x_1+2L_1,x_2)=\widetilde{v}(x)=\widetilde{v}(x_1,x_2+2L_2)\quad \forall x\in \xR^2.\label{n210c}
\end{align}
Similarly, when $d=1$, 
$\widetilde{v}\colon \xR\rightarrow \xR$ is defined by
\begin{align}
&\widetilde{v}(x)=v(x)\quad \forall x\in Q,\\
&\widetilde{v}(-x)=\widetilde{v}(x) \quad \forall x\in \xR,\label{n211b}\\
&\widetilde{v}(x+2L_1)=\widetilde{v}(x)\quad \forall x\in \xR.\label{n211c}
\end{align}
\end{defi}
\begin{defi}\label{D:S}
Given $\sigma\in \xR$, 
denote by $H^\sigma_\pair(\xT^d)$ the Sobolev space of those periodic functions which are even (satisfying \e{n210b}-\e{n210c} when 
$d=2$ and \e{n211b}-\e{n211c} for $d=1$). 
\end{defi}
Now consider the case $d=1$ (to fix notations) and $u\in H^\sigma_\pair(\xT)$ with $\sigma>d/2+1=3/2$. 
Then, $\partial_{x}u(x)$ is $C^0$ and odd 
which implies that 
$\partial_{x} u(0)=0$. 
Moreover, one has $u(L_1+\eps)=u(-L_1+\eps)=u(L_1-\eps)$ and hence one has also 
$\partial_{x}u(L_1)=0$ 
(then $\partial_{x}u(nL_1)=0$ for any $n\in\xZ$). We have a similar result when $d=2$. This proves that
\be\label{n247}
\forall \sigma>\frac{d}{2}+1,~\forall v\in H^\sigma_e(\xT^d),\qquad 
v\arrowvert_Q\in H^\sigma_e(Q).
\ee
Conversely, the following result shows that any function $v$ in $H^\sigma_e(Q)$, with $\sigma\in (3/2,7/2)$, 
is the restriction to $Q$ of a function belonging to 
$H^\sigma_e(\xT^d)$.

\begin{prop}[from Prop.~$6.5$ in \cite{ABZ4}]\label{P3}
Let $1\le d\le 2$ and $\frac{3}{2} <\sigma < \frac 7 2$. 
Then the map $v\mapsto \widetilde{v}$ is continuous from $H^\sigma_e(Q)=\{v\in H^\sigma(Q): \partial_\nu v=0\text{ on }\partial Q\}$  to $H^\sigma_{\pair}(\xT^d)$.
\end{prop}

We are now in position to define $G(\eta)\psi$ when $\eta$ and $\psi$ belong to some space $H^\sigma_e(Q)$.
To do so, let $\widetilde{\eta},\widetilde{\psi}$ be as given by Definition~\ref{De:tilde}. If $3/2<\sigma<7/2$, 
it follows from Proposition~\ref{P3} that $\widetilde{\eta},\widetilde{\psi}$ belong to $H^\sigma_e(\xT^d)$. 
If one further assumes that $\sigma>d/2+1$, as recalled in \S\ref{S:22}, there exists a unique periodic variational solution to
\be\label{n241a}
\begin{aligned}
&\Delta_{x,y}\widetilde\phi=0\quad\text{in }\{(x,y)\in \xR^d\times\xR\, :\, -h\le y\le\widetilde\eta(x)\},\\
&\widetilde\phi(x,y)\arrowvert_{y=\widetilde\eta(x)}=\widetilde\psi(x), 
\quad \partial_y\widetilde\phi\arrowvert_{y=-h}=0.
\end{aligned}
\ee

\begin{defi}\label{D:8}
Consider $\eta$ and $\psi$ in $H^\sigma_e(Q)$ with $d/2+1<\sigma<7/2$. 
We define $G(\eta)\psi$ (resp.\ $\phi$) by taking the restriction to $Q$ (resp.\ $\Omega$):
\be\label{n170a}
\phi=\widetilde{\phi}\arrowvert_{\Omega},\quad G(\eta)\psi=G(\widetilde{\eta})\widetilde{\psi}\arrowvert_{Q}. 
\ee
where $\Omega= \{(x,y)\in Q\times\xR\, :\, -h\le y\le\eta(x)\}$.
\end{defi}

\begin{prop}\label{P6B}
Assume that $(\eta,\psi)\in H^{\sigma}_{\pair}(Q)\times H^\sigma_{\pair}(Q)$ with 
$\sigma\in (d/2+2,7/2)$ where $d=1,2$. Then 
\be\label{n353}
\nabla_{x,y}\phi\in C^1(\overline{\Omega}),\quad
\partial_n\phi=0\text{ on }\partial\mathscr{R}\cap \partial\Omega,
\ee
and
\be\label{n354}
G(\eta)\psi\in H^{\sigma-1}_e(Q).
\ee
\end{prop}
\begin{proof}
Let us prove \e{n353}. The fact that $\nabla_{x,y}\phi\in C^1(\overline{\Omega})$ follows from Proposition~\ref{P6}. 
It remains only to prove that 
$\partial_n\phi=0$ on $\partial\mathscr{R}\cap \partial\Omega$. By definition of $\phi$, 
we have to prove that $\partial_n\widetilde{\phi}(x,y)=0$ for any $(x,y)\in \overline{\widetilde{\Omega}}$ with 
$x\in \partial Q$. To simplify notations, assume that $d=1$ and $L_1=1$. Then, as 
already mentioned after the statement of Definition~\ref{D:S}, 
notice that $\widetilde\eta(x)=\widetilde\eta(-x)$ and $\widetilde\eta(L_1-x)=\widetilde\eta(L_1+x)$ 
since $\widetilde\eta$ is $2L_1$-periodic and even. 
Since $\sigma>3/2$, one has $\widetilde\eta\in H^\sigma_e(\xT)\subset C^1(\xT)$ and one deduces that
$\px\widetilde\eta(0)=0=\px\widetilde\eta(1)$. 
Similarly $\widetilde\psi(x)=\widetilde\psi(-x)$, $\widetilde\psi(L_1-x)=\widetilde\psi(L_1+x)$ and, by uniqueness of the harmonic extension, 
$\widetilde\phi$ satisfies the same symmetries in $x$. Consequently, $\px\widetilde\phi$ 
is odd in $x$ and, 
since $\px\widetilde\phi$ is continuous, we infer that
$$
\px\widetilde\phi(0,y)=0 \quad \forall y\in [-h,\eta(0)],\quad 
\px\widetilde\phi(L_1,y)=0 \quad \forall y\in [-h,\eta(L_1)].
$$
This completes the proof of \e{n353}. 

It remains to prove \e{n354}. It follows from Proposition~\ref{P7}Ê
and Proposition~\ref{P3} 
that 
$G(\widetilde\eta)\widetilde\psi\in H^{\sigma-1}(\xT^d)$. Therefore, in view of \e{n247}, it remains only to prove 
that $G(\widetilde\eta)\widetilde\psi$ is even (see Definition~\ref{D:S}). This in turn follows from the definition 
$G(\widetilde\eta)\widetilde\psi=\py\widetilde\phi-\nabla\widetilde\eta\cdot\nabla\widetilde\phi\big\arrowvert_{y=\widetilde\eta}$ 
and the symmetries of $\widetilde{\phi}$ that we already used.
\end{proof}

\subsection{The Cauchy problem in a rectangular tank}
We are now in position to define what is a solution of the equations in a rectangular tank and to prove the existence and uniqueness of such solutions.

\begin{defi}\label{D:10}
Let $d\in \{1,2\}$ and $s>d/2+2$ and consider
$$
(\eta,\psi)\in C^0([0,T];H^s_e(Q)\times H^s_e(Q)).
$$
Then $(\eta,\psi)$ is a solution of the Craig--Sulem--Zakharov system \e{system} with initial data $(\eta_0,\psi_0)$ 
if and only if the extended functions $(\widetilde{\eta},\widetilde\psi)$ satisfies
\begin{equation}\label{system-per}
\left\{
\begin{aligned}
&\partial_t \widetilde\eta=G(\widetilde\eta)\widetilde\psi,\\
&\partial_t \widetilde\psi + g\widetilde\eta+ \frac{1}{2}|\nabla\widetilde\psi|^2
-\frac{1}{2(1+| \nabla\widetilde\eta|^2)}\bigl(G(\widetilde\eta)\widetilde\psi+\nabla\widetilde\eta\cdot \nabla\widetilde\psi\bigr)^2= 0,\\
&(\widetilde\eta,\widetilde\psi)\arrowvert_{t=0}=(\widetilde\eta_0,\widetilde\psi_0).
\end{aligned}
\right.
\end{equation}
\end{defi}
\begin{prop}\label{T2}
Let $d\in \{1,2\}$ and $s\in (d/2+2,7/2)$. Consider initial data $\eta_0,\psi_0$ in $H^s_e(Q)$ and 
denote by 
$\widetilde{\eta}_0,\widetilde{\psi}_0$ 
the extensions as given by Definition~$\ref{De:tilde}$. 
If
\be\label{n10a-bis}
\widetilde B_0=
\frac{G(\widetilde{\eta}_0)\widetilde{\psi}_0+\nabla\widetilde{\eta}_0\cdot\nabla\widetilde{\psi}_0}
{1+|\nabla\widetilde{\eta}_0|^2}\in H^{s-\mez}(\xT^d),\quad
\widetilde V_0=\nabla\widetilde{\psi}_0-\widetilde B_0\nabla\widetilde{\eta}_0\in H^{s-\mez}(\xT^d),
\ee
then there exists $T>0$ and a unique 
solution $(\widetilde{\eta},\widetilde\psi)\in C^{0}\big([0,T];H^{s}_{\pair}(\xT^d)\times H^{s}_{\pair}(\xT^d)\big)$ to 
the Cauchy problem \e{system-per} such that 
$$
(\widetilde B,\widetilde V)
\in C^{0}\big([0,T]; H^{s-\mez}(\xT^d)\times H^{s-\mez}(\xT^d)\big),
$$
where
$$
\widetilde B=\frac{G(\widetilde{\eta})\widetilde\psi+\nabla \widetilde{\eta}\cdot\nabla\widetilde\psi}{1+|\nabla\widetilde{\eta}|^2},\quad
\widetilde V=\nabla\widetilde\psi-\widetilde B\nabla\widetilde \eta.
$$
\end{prop}
\begin{proof}In view of Theorem~\ref{T2a} it remains only to prove that 
$\widetilde\eta$ and $\widetilde\psi$ 
are even in $x_k$ for $1\le k\le d$ (satisfying \e{n210b}-\e{n210c} when 
$d=2$ and \e{n211b}-\e{n211c} for $d=1$). 
To do so, assume that $d=2$ and $k=1$. Set 
$\eta^\sharp(t,x)=\widetilde\eta(t,-x_1,x_2)$, 
$\psi^\sharp(t,x)=\widetilde\psi(t,-x_1,x_2)$. 
Since the equations are invariant by the symmetry $x_1\mapsto -x_1$, we get that 
$(\eta^\sharp,\psi^\sharp)$ satisfies the same equations. Moreover, one has 
$\eta^\sharp\arrowvert_{t=0}=\widetilde{\eta}_0$, $\psi^\sharp\arrowvert_{t=0}=\widetilde{\psi}_0$ since 
$\widetilde{\eta}_0$ and $\widetilde{\psi}_0$ are even in $x_k$. 
By uniqueness, we deduce that $\eta^\sharp=\widetilde\eta$ and $\psi^\sharp=\widetilde\psi$, which is the desired property.
\end{proof}
\begin{rema}Let us give three cases where the assumptions are satisfied. 
$i)$ $\eta_0,\psi_0$ are finite linear combinations of terms of the form \e{n302}.

$ii)$ Consider the case where $\psi_0=0$. Then $\widetilde{\psi}_0$ vanishes and hence 
$\widetilde B_0=0$ and $\widetilde V_0=0$ so Assumption~\e{n10a-bis}Ê
holds whenever $\widetilde{\eta}_0\in H^s(\xT^d)$ (and hence for any $\eta_0\in H^s_e(Q)$). 

$iii)$ Assume that $d=1$ and $\eta_0,\psi_0\in H^\sigma_e(Q)$ for some $\sigma\in (3,7/2)$. 
Proposition~\ref{P3} implies that $\widetilde{\eta}_0$ and $\widetilde{\psi}_0$ belong to $H^{\sigma}_{\pair}(\xT^d)^2$. Since $\sigma>3$, 
it follows from Proposition~\ref{P7} below and the usual nonlinear estimates in Sobolev spaces that 
$\widetilde B_0$ and $\widetilde V_0$ are in $H^{\sigma-1}(\xT^d)$. 
So we may apply the assumptions of Corollary~\ref{T2}Ê
with $s=\sigma-1/2$.
\end{rema}

\section{Pohozaev identity}\label{S:Pohozaev}

In this section the time is seen as a parameter and we skip it. As above, we 
denote by $d\in \{1,2\}$ 
the dimension of the free surface. 
For a 3D (resp.\ 2D) fluid one has $d=2$ (resp.\ $d=1$) 
and we use the notation 
$\nabla=(\partial_{x_1},\partial_{x_2})$ (resp.\ $\nabla=\px$). 
Our goal is to prove a Pohozaev type identity for 
$G(\eta)\psi$, that is we want to compute 
$$
\int_Q (G(\eta)\psi)(\poh \psi)\, dx.
$$
Recall from the previous section that $G(\eta)\psi$ is defined by taking the restrictions to $Q$ of a periodic function 
$G(\widetilde \eta)\widetilde \psi$ 
(see Definition~\ref{D:8}). Consequently, even if we are working in a bounded domain, we are essentially handling 
periodic functions defined on $\xR^d$.  Recall also that one has 
$\partial_n\phi=0$ on $\partial\mathscr{R}\cap \partial\Omega$ (see Proposition~\ref{P6B}). 

%The main result about the Dirichlet to Neumann operator that we use 
%in that paper is the following proposition. 

\begin{prop}[Pohozaev identity]\label{Pohozaev}
Assume that $(\eta,\psi)$ belongs to $H^{\sigma}_{\pair}(Q)\times H^\sigma_{\pair}(Q)$ for some $\sigma>d/2+2$. 
Denote by $R$ the solid part of $\partial\Omega$:
$$
R\defn \partial\mathscr{R}\cap\partial\Omega,
$$
($\mathscr{R}=Q\times [-h,+\infty)$) and denote by $n$ the unit outward normal to~$\partial\Omega$.
Then, %the following identity holds
\be\label{n21}
\begin{aligned}
&\int_Q (G(\eta)\psi)(\poh \psi)\, dx\\
&\qquad =\mez\int_{R}\la \nabla_{x,y}\phi\ra^2 \begin{pmatrix}x \\ y\end{pmatrix}\cdot n\, dS-\frac{d-1}{2}\iint_\Omega \la \nabla_{x,y}\phi\ra^2\, dxdy\\
&\qquad \quad + \mez \int_Q(\eta-\poh\eta)\big[ V^2+B^2 -2 BG(\eta)\psi\big]\, dx,
\end{aligned}
\ee
where $\phi$ is given by \e{n170a}, 
$B=(\py\phi)\arrowvert_{y=\eta(x)}$ and $V=(\nabla_x\phi)\arrowvert_{y=\eta(x)}$ (Proposition~\ref{P6B} implies that 
$\nabla_{x,y}\phi\in C^1(\overline{\Omega})$ and hence all the terms are well-defined). 
\end{prop}
\begin{rema}\label{R32}
$i)$ If $d=1$ then the second term in the right-hand side of \e{n21} vanishes and, since $n\cdot\nabla_{x,y}\phi=0$ on $R$, the first one simplifies to
$$
\mez\int_{R}\la \nabla_{x,y}\phi\ra^2 \begin{pmatrix}x \\ y\end{pmatrix}\cdot n\, dS
=\frac{L_1}{2} \int_{-h}^{\eta(L_1)} (\py\phi (L_1,y))^2\, dy +\frac{h}{2} \int_0^{L_1} (\px\phi(x,-h))^2\, dx.
$$
Consider now the case $d=2$. Then one has also $(x,y)\cdot n\ge 0$. Indeed,
\begin{alignat*}{5}
&\text{on } \{x_1=L_1\} \quad&&\text{one has } && n=(1,0,0) \quad &&\text{and }&&(x,y)\cdot n = L_1,\\
&\text{on } \{x_2=L_2\} \quad&&\text{one has }  && n=(0,1,0) \quad &&\text{and } &&(x,y)\cdot n = L_2,\\
&\text{on } \{y=-h\} \quad&&\text{one has }  && n=(0,0,-1) \quad &&\text{and }&&(x,y)\cdot n = h,
\end{alignat*}
and moreover, $(x,y)\cdot n\equiv 0$ on the two other faces $\{x_1=0\}$ and $\{x_2=0\}$. 

$ii)$ Another Pohozaev identity for the fractional Laplacian has been proved by Ros-Oton and Serra \cite{Ros-Oton-Serra}. Moreover, Biccari \cite{Biccari} deduced from the Pohozaev identity in \cite{Ros-Oton-Serra} a Pohozaev identity for solutions to fractional Schr\"odinger equations. 
Namely, it is proved in \cite{Ros-Oton-Serra} 
that, if $Q$ is any $C^{1,1}$ domain of $\xR^d$, $s\in (0,1)$ and $u\in H^s(\xR^d)$ vanishes 
in $\xR^d\setminus Q$, then 
\be\label{nROS}
\begin{aligned}
\int_Q  (\poh u)(-\Delta)^su\, dx &=\frac{2s-d}{2}\int_Q u(-\Delta)^s u\, dx\\
&\quad-\frac{\Gamma(1+s)^2}{2}
\int_{\partial Q} \left( \frac{u}{\dist(x,\partial Q)^s}\right)^2 (x\cdot \nu)\, dS.
\end{aligned}
\ee
To compare both results the important think to note is that, in the case without boundary, one has $G(0)=(-\Delta)^{1/2}$. 
However, \e{n21} and \e{nROS} involve functions satisfying different boundary conditions. 
Another essential difference for our purpose is that \e{n21} applies in the variable coefficients case 
where $\eta\neq 0$. Also the proofs of \e{n21} and \e{nROS} are different. 
The proof of \e{n21} given below is in fact guided by the study of the commutator $[x\px,G(\eta)]$ in 
Alazard-Delort \cite[Chapter $4$]{AlDe-Sob}. 
However, one cannot apply 
the results of \cite{AlDe-Sob} because of the boundary conditions on $R$ (and also because we consider the case 
$d\ge 1$ while the analysis in \cite{AlDe-Sob}Ê
is restricted to $d=1$). 
Compared to \cite[Chapter $4$]{AlDe-Sob}, the main new result here is the observation that the contribution of these boundary conditions is given by a positive term  
(namely the first term in the right-hand side of \e{n21}).
\end{rema}
\begin{proof}
The proof of this proposition relies on the divergence theorem applied to a well chosen vector field.
Introduce the scalar function 
$$
\theta\defn x\cdot\nabla_x\phi+y\py\phi
$$
and the vector field
$$
X=\theta\nabla_{x,y}\phi.
$$
We are going to compute the integral of $\cn_{x,y} X$ by two different ways. The wanted 
identity \e{n21}Ê
will be deduced by comparing the two results. 

Proposition~\ref{P6B} implies that 
$\nabla_{x,y}\phi\in C^1(\overline{\Omega})$ and 
$\partial_n\phi=0$ on $\partial\mathscr{R}\cap \partial\Omega$. 
This will allow us to justify all the computations done below. 

\noindent{\em First computation.} We want to exploit the fact that, 
since $\partial_n\phi=0$ on $R$, one has $X\cdot n=0$ on $R$. To do so we begin by writing
$$
\iint_\Omega \cn_{x,y} X\, dxdy=\int_{\partial\Omega} X\cdot n\, dS=\int_{\partial\Omega\setminus R}X\cdot n\, dS.
$$
Since $\partial\Omega\setminus R=\{(x,y),x\in Q,y=\eta(x)\}$, by definition of $G(\eta)\psi$, 
the previous identity simplifies to
\begin{align*}
\iint_\Omega \cn_{x,y} X\, dxdy&=\int_{\partial\Omega\setminus R} \theta \partial_n\phi\, dS\\
&=\int_Q \theta(x,\eta)\sqrt{1+|\nabla\eta|^2}\, \partial_n \phi \arrowvert_{y=\eta}\, dx\\
&=\int_Q \theta(x,\eta) G(\eta)\psi\, dx.
\end{align*}

Now, write
$$
\nabla_x\psi=\nabla_x(\phi(x,\eta(x)))=(\nabla_x\phi)(x,\eta(x))+(\py \phi)(x,\eta(x))\nabla_x\eta.
$$
Since $B=\partial_y\phi(x,\eta)$, we get 
that $(\nabla_x\phi)(x,\eta(x))=\nabla_x\psi-B\nabla_x\eta$. By definition of $\theta$, we deduce that
$$
\theta(x,\eta)=x\cdot(\nabla_x\psi-B\nabla_x\eta)+\eta B=x\cdot\nabla_x\psi+(\eta-x\cdot\nabla_x\eta)B.
$$
We thus end up with
\be\label{n271}
\iint_\Omega \cn_{x,y} X\, dxdy=\int_Q (G(\eta)\psi)(\poh\psi)\, dx+\int_Q(\eta-x\cdot\nabla_x\eta)B G(\eta)\psi\, dx.
\ee

\noindent{\em Second computation.} Set
$$
\mathcal{W}=\la \nabla_{x,y}\phi\ra^2.
$$
As can be verified by a direct computation, one has
$$
\cn_{x,y} X=\mathcal{W}+\mez x\cdot\nabla_x \mathcal{W}+\mez y\py\mathcal{W},
$$
and hence
$$
\cn_{x,y} X=\cnx\left(\frac{\mathcal{W}}{2}x\right)+\py\left(\frac{\mathcal{W}}{2}y\right)-\frac{d-1}{2}\mathcal{W}.
$$
Introduce the vector field
$$
Y=\frac{\mathcal{W}}{2}\begin{pmatrix}x \\ y\end{pmatrix}.
$$
Then the previous identity reads $\cn_{x,y} X=\cn_{x,y} Y-\frac{d-1}{2}\mathcal{W}$. Consequently
$$
\iint_\Omega \cn_{x,y} X\, dxdy=\int_{\partial\Omega} Y\cdot n\, dS-\frac{d-1}{2}\iint_\Omega\mathcal{W}\, dxdy.
$$
Now observe that,
\begin{align*}
&\int_{R}Y\cdot n\, dS=\mez\int_{R}\la \nabla_{x,y}\phi\ra^2 \begin{pmatrix}x \\ y\end{pmatrix}\cdot n\, dS,\\
&\int_{\partial\Omega\setminus R}Y\cdot n\, dS
=\mez\int_{Q} \mathcal{W}\arrowvert_{y=\eta} \begin{pmatrix}x \\ \eta\end{pmatrix}\cdot\begin{pmatrix}-\nabla\eta\\1\end{pmatrix}
\, dx=\mez\int_Q (\eta-\poh\eta)\mathcal{W}\arrowvert_{y=\eta}\, dx.
\end{align*}
Therefore
\begin{align*}
\iint_\Omega \cn_{x,y} X\, dxdy&=
\mez\int_{R}\la \nabla_{x,y}\phi\ra^2 \begin{pmatrix}x \\ y\end{pmatrix}\cdot n\, dS\\
&\quad+\mez\int_Q (\eta-\poh\eta)\mathcal{W}\arrowvert_{y=\eta}\, dx
-\frac{d-1}{2}\iint_\Omega\mathcal{W}\, dxdy.
\end{align*}
By combining this identity with \e{n271} we obtain the wanted result~\e{n21}.
\end{proof}

\section{The main identity}\label{S:4}

\begin{theo}\label{main}
Consider a solution
$(\eta,\psi)\in C^{0}([0,T];H^{s}_{\pair}(Q)\times H^{s}_{\pair}(Q))$ 
of the Craig--Sulem--Zakharov system \e{system} with $T>0$, 
$s>d/2+2$, $d\in\{1,2\}$ (see Definition~$\ref{D:10}$). Set
$$
\Theta=-\eta\partial_t\psi-\frac{g}{2}\eta^2,
$$
and 
$$
\mathscr{B}(T)=L_1L_2 \int_0^T
\left(\frac{1}{L_2}\int_0^{L_2} \Theta(t,L_1,x_2)\, dx_2+\frac{1}{L_1}\int_0^{L_1} \Theta(t,x_1,L_2)\, dx_1\right)\, dt.
$$

Then the following identity holds
\be\label{n301}
\mathscr{B}(T)=\frac{T}{2}\mathcal{H}
+P+I_1+I_2+I_3,
\ee
where $\mathcal{H}$ is the energy
$$
\mathcal{H}=\frac{1}{2} \int_Q\big[\psi G(\eta)\psi+g\eta^2\big]\, dx,
$$
$P$ is a positive integral (see Remark~$\ref{R32}$) given by
$$
P\defn\mez\int_0^T\int_{\partial\mathscr{R}\cap\partial\Omega}\la \nabla_{x,y}\phi\ra^2 \begin{pmatrix}x \\ y\end{pmatrix}\cdot n\, dS\, dt
%P_L&=\mez\iint_{\Gamma\times[0,T]} \la \nabla_{x,y}\phi(t,x,\eta(t,x))\ra^2
%\, x\cdot N\, dSdt\qquad (N\text{ is the normal to }\Gamma),\\
%P_B&=\frac{h}{2}\iint_{Q\times [0,T]} |\nabla_x\phi(t,x,-h)|^2\, dx\, dt,
$$
and $I_\ell$ are integrals denoting remainder terms:
\begin{align*}
I_1&=\frac{5+2d}{8}\iint_{Q\times [0,T]} \eta|\nabla_x\phi|^2(t,x,-h)\, dxdt,\\[1ex]
I_2&=-\frac{5+2d}{4}\int_0^T\iint_{\Omega(t)} (\py\phi)(\nabla_x\eta\cdot\nabla_x\phi) \, dydxdt,\\[1ex]
I_3&=-\left(\frac{d}{2}-\frac{1}{4}\right)\int_Q \eta \psi\,dx \Big\arrowvert_{t=0}^{t=T}-\int_Q \eta(\poh\psi)\, dx\Big\arrowvert_{t=0}^{t=T}.
\end{align*}
\end{theo}

Before proving this result, let us deduce the following corollary. 

\begin{coro}\label{C2d2}
Consider a solution
$(\eta,\psi)\in C^{0}([0,T];H^{s}_{\pair}(Q)\times H^{s}_{\pair}(Q))$ 
of the Craig--Sulem--Zakharov system \e{system} for some $T>0$ and 
$s>d/2+2$ with $d\in\{1,2\}$. Assume that
\be\label{n9c-bis}
\eta\ge -\frac{4h}{9},\quad \int_Q\eta \, dx=0,
\ee
and that there exists two positive constants $A,B$ such that
\be\label{n1025}
B<\frac{2}{5+2d},\quad T\ge \frac{4}{2-(5+2d)B}\left[1+\frac{(2d+3)\max\{L_1,L_2\}}{\sqrt{g}}A\right],
\ee
(where $\max\{L_1,L_2\}=L_1$ if $d=1$) and
\be\label{n1055}
\sup_{t\in [0,T]}\lA \nabla\eta(t)\rA_{L^\infty}\le B,\quad 
\sup_{t\in[0,T]}\lA \nabla\psi(t)\rA_{L^2}\le A\sqrt{2\mathcal{H}}.
\ee
Then
$$
\mathscr{B}(T) \ge \mathcal{H}.
$$
\end{coro}
\begin{proof}The proof is similar to the proof of Corollary \ref{C2}. 
Firstly, notice that
$$
P\ge \frac{h}{2}\iint_{Q\times [0,T]} |\nabla_x\phi|^2(t,x,-h)\, dxdt.
$$
Therefore $P+I_1\ge 0$ provided that $h/2+((5+2d)/8)\eta\ge 0$, which holds true 
if $\eta\ge -4h/9$ and $d=1,2$. 

On the other hand, 
$$
\la I_2\ra\le \frac{5+2d}{4}\sup \lA \nabla_x\eta(t,\cdot)\ra_{L^\infty} 
\int_0^T\iint_{\Omega(t)} \la \py\phi \ra \la \nabla_x\phi\ra \, dydxdt
\le \frac{5+2d}{4} B T \mathcal{H}.
$$
Using the Cauchy-Schwarz and Poincar\'e inequalities, one has
\begin{align*}
\la\int_Q \eta \psi\,dx\ra &\le \max\{L_1,L_2\}\lA \eta\ra_{L^2} \lA \nabla_x\psi\rA_{L^2},\\
\la\int_Q \eta \poh\psi\,dx\ra &\le \max\{L_1,L_2\}\lA \eta\ra_{L^2} \lA \nabla_x\psi\rA_{L^2},
\end{align*}
Since $\lA \eta\rA_{L^2}\le \sqrt{2\mathcal{H}/g}$, we deduce from the assumption~\ref{n1055} 
that
$$
\la I_3\ra\le \left(d+\tdm \right)\max\{L_1,L_2\}\frac{2}{\sqrt{g}}A \mathcal{H}.
$$
Therefore  it follows from \e{n301} that
\be\label{n501}
\mathscr{B}(T)\ge \left(\frac{T}{2}-\frac{5+2d}{4} B T -\left(d+\tdm \right)\max\{L_1,L_2\}\frac{2}{\sqrt{g}}A\right)\mathcal{H}
,
\ee
and hence $\mathscr{B}(T)\ge \mathcal{H}$ by assumption on $T$.
\end{proof}
\begin{nota*}
We write simply
$$
\int \, dx ,\quad \int \, dy ,\quad  
\int\,dt
$$
as shorthand notations for, respectively,
$$
\int_Q \, dx ,\quad \int_{-h}^{\eta(t,x)} \, dy ,\quad 
\int_0^T \,dt.
$$
\end{nota*}

The proof of Theorem~\ref{main} will be made in three steps. 
First, we exploit the Pohozaev identity obtained in the previous section.

\begin{lemm}
There holds
\be\label{n30}
\begin{aligned}
\iint(\partial_t \eta)(\poh\psi) \, dxdt&=P-\frac{d-1}{2}\iiint \la \nabla_{x,y}\phi\ra^2\, dydxdt\\
&\quad+\iint (\poh \eta-\eta)(\partial_t\psi+g\eta)\, dxdt.
\end{aligned}
\ee
\end{lemm}
\begin{proof}
Since $\partial_t\eta=G(\eta)\psi$, it follows from Proposition~\ref{Pohozaev} that
\begin{align*}
\iint (\partial_t \eta)(\poh\psi)\, dxdt&=P-\frac{d-1}{2}\iiint \la \nabla_{x,y}\phi\ra^2\, dydxdt\\
&\quad+\mez\iint(\eta-\poh\eta)(V^2+B^2-2BG(\eta)\psi)\, dxdt.
\end{align*}
To prove \e{n30}, it is thus sufficient to prove that
\be\label{B2-1}
\mez (V^2+B^2-2BG(\eta)\psi)=-\partial_t\psi-g\eta.
\ee
To obtain this identity, we use the fact that 
\begin{equation*}
B=\frac{G(\eta)\psi+\nabla\eta\cdot\nabla\psi}{1+|\nabla\eta|^2},\quad
V=\nabla\psi-B\nabla\eta.
\end{equation*}
Consequently, $G(\eta)\psi=B-V\cdot\nabla\eta$ and 
\begin{equation}\label{n25}
\mez |\nabla \psi|^2 -\mez \frac{\left( \nabla\eta\cdot\nabla\psi 
+ G(\eta)\psi \right)^2}{1+|\nabla\eta|^2} 
= \mez V^2 +B V\cdot\nabla\eta- \mez B^2,
\end{equation}
so~\eqref{B2-1} follows from the second equation of \eqref{system}.
\end{proof}

Next we integrate by parts and use the equations for $(\eta,\psi)$ to simplify the expressions.

\begin{nota}
Let $f\colon Q\times[0,T]\rightarrow \xR$ and $d=1,2$. If $d=2$ then we set
$$
\Gamma=\big([0,L_1]\times \{L_2\}\big)\times (\{L_1\}\times [0,L_2]\big),
$$
and use the notation
$$
\iint_{\Gamma\times[0,T]} f\,dS dt
=L_1L_2\int_0^T\left[\frac{1}{L_1}\int_0^{L_1} f(t,x_1,L_2)\, dx_1+\frac{1}{L_2}\int_0^{L_2}f(t,L_1,x_2)\, dx_2\right]\, dt.
$$
Similarly, when $d=1$ and $f\colon [0,L_1]\times [0,T]\rightarrow 0$, we write
$$
\iint_{\Gamma\times[0,T]} f\,dSdt=L_1\int_0^T f(t,L_1)\, dt.
$$
With this notation, one has
$$
\mathscr{B}(T)=\iint_{\Gamma\times[0,T]} \Theta \,dSdt.
$$
\end{nota}

\begin{lemm}
%Let $\Theta=-\eta\partial_t\psi-\frac{g}{2}\eta^2$. 
There holds
\be\label{n35}
\begin{aligned}
\iint_{\Gamma\times[0,T]} \Theta \,dSdt&=P-\frac{d-1}{2}\iiint \la \nabla_{x,y}\phi\ra^2\, dydxdt\\
&\quad-(1+d)\iint \eta\partial_t\psi\, dxdt-\frac{2+d}{2} \iint g \eta^2\, dx dt
\\
&\quad -\int \eta(\poh\psi)\, dx\Big\arrowvert_{t=0}^{t=T}.
\end{aligned}
\ee
\end{lemm}
\begin{proof}
Notice that for any scalar function $f$, one has
$$
\iint \cn_x(f x)\, dx dt=\iint_{\Gamma\times[0,T]} f\,dSdt.
$$
Then, integrating by parts, we find that
\begin{align*}
\iint(\poh\eta-\eta)(\partial_t\psi+g\eta)\, dxdt&=\mez\iint_{\Gamma\times[0,T]}g\eta^2\,dSdt-\frac{2+d}{2}\iint g\eta^2\, dxdt\\
&\quad+\iint_{\Gamma\times[0,T]}\eta\partial_t\psi\, dSdt-(1+d)\iint \eta \partial_t\psi\, dxdt\\
&\quad+\iint (\partial_t \eta)(\poh\psi)\,dxdt-\int \eta(\poh\psi)\, dx\Big\arrowvert_{t=0}^{t=T}
.
\end{align*}
We obtain the wanted result by combining the previous identity with \e{n30}.
\end{proof}

\begin{rema}\label{R:46}
Consider the case when $d=1$ and set 
$m(t)=\eta(t,L_1)$. Then, 
Since $\px\eta(t,L_1)=0$ and $\px\psi(t,L_1)=0$, it follows from \e{system} that
$$
\partial_t\psi(t,L_1)=-g\eta(t,L_1)+\mez (G(\eta)\psi)(t,L_1)^2.
$$
Since $G(\eta)\psi=\partial_t\eta$, we conclude that
\be\label{n94}
\partial_t\psi(t,L_1)=-gm(t)+\mez m'(t)^2.
\ee
This shows that $\Theta=\mez\big[gm(t)^2-m(t)m'(t)^2\big]$.
\end{rema}
The next step consists in computing the right-hand side of \e{n35}. 
The wanted result \e{n301}Ê
will be a direct consequence of the previous results and the 
following lemma. 

\begin{lemm}
Set
$$
A\defn (1+d)\iint \eta \partial_t \psi \,dxdt+\frac{2+d}{2}\iint g\eta^2\,dxdt+\frac{d-1}{2}\iiint \la \nabla_{x,y}\phi\ra^2\, dydxdt.
$$
There holds
\be\label{n151}
\begin{aligned}
A&=-\frac{T}{4}\int \big[\psi G(\eta)\psi+g\eta^2\big]\, dx+\left(\frac{d}{2}-\frac{1}{4}\right)\int \eta \psi\,dx \Big\arrowvert_{t=0}^{t=T}\\
&\quad-\frac{5+2d}{8}\iint \eta|\nabla_x\phi|^2(t,x,-h)\, dxdt\\
&\quad+\frac{5+2d}{4}\iiint (\py\phi)(\nabla_x\eta\cdot\nabla_x\phi) \, dydxdt.
\end{aligned}
\ee
\end{lemm}
\begin{proof}
Let $\alpha$ be a parameter to be determined. We split $A$ as $A=A_1+A_2$ where
\begin{align*}
A_1&=\alpha\iint \eta \partial_t \psi\, dxdt+\frac{d-1}{2}\iiint \la \nabla_{x,y}\phi\ra^2\, dydxdt,\\ 
A_2&= \left(1+d-\alpha\right)\iint \eta \partial_t \psi\, dxdt+\frac{2+d}{2} \iint g\eta^2\,dxdt.
\end{align*}
Integrate by parts and use the equation $\partial_t\eta=G(\eta)\psi$ to obtain
$$
-\alpha \iint \eta \partial_t \psi\, dxdt=-\alpha\iint \psi G(\eta)\psi\, dxdt+\alpha \int \eta \psi\, dx\Big\arrowvert_0^T.
$$
and
$$
\frac{d-1}{2}\iiint \la \nabla_{x,y}\phi\ra^2\, dydxdt=\frac{d-1}{2}\iint \psi G(\eta)\psi\, dxdt,
$$
so that
\be\label{n47a}
A_1=\left(\frac{d-1}{2}-\alpha\right) \iint \psi G(\eta)\psi\, dxdt+\alpha \int \eta \psi\, dx\Big\arrowvert_0^T.
\ee
On the other hand, directly from the equation for $\psi$ (see \e{system}) and the definition of $B$ (see \e{n10b}), one has
\be\label{n48}
A_2=\left(\alpha-\frac{d}{2}\right)\iint g\eta^2+\left(1+d-\alpha\right) \iint \eta\left(-\mez|\nabla\psi|^2+\mez 
(1+|\nabla\eta|^2)B^2\right).
\ee
In addition, it is easily verified (see \e{n25} and \e{n10b}) that
\be\label{n48x1}
-\mez|\nabla\psi|^2+\mez(1+|\nabla\eta|^2)B^2=\mez \Big[(\py\phi)^2-|\nabla\phi|^2-2(\py\phi)(\nabla\phi)\cdot\nabla\eta\Big]\Big\arrowvert_{y=\eta}.
\ee

The next calculations rely in a crucial way on 
the analysis done by Benjamin and Olver (\cite{BO}) of the conservation laws for water waves. 
This means that, 
to compute $A_2$, we will use two elementary identities. Consider a scalar function $u=u(x,y)$ and 
a vector field $f=f(x,y)$ with values in $\xR^2$ (the time is seen as a parameter and we skip it). Firstly, one has
\be\label{n40}
\int_Q u(x,\eta(x))\, dx=\int_Q \int_{-h}^{\eta(x)}\partial_y u(x,y)\, dy 
+\int_Q u(x,-h)\, dx.
\ee
Similarly,
\be\label{n41}
\int_Q f(x,\eta)\cdot\nabla\eta\, dx+\int_Q \int_{-h}^{\eta}\cnx f\, dydx=
\int_{(\partial Q\times [-h,+\infty))\cap \partial\Omega} f\cdot \nu \, dS.
\ee
Then
\begin{multline}\label{n41ab}
\int u(x,\eta(x))\, dx+\int f(x,\eta)\cdot\nabla\eta\, dx
\\=\iint(\py u-\cnx f)\, dydx
+\int u(x,-h)\, dx
+\int_{(\partial Q\times [-h,+\infty))\cap \partial\Omega} f\cdot \nu \, dS.
\end{multline}
Recall that $\Delta_{x,y}\phi=0$ and that $\partial_n\phi$ vanishes on $\partial\mathscr{R}\cap\partial \Omega$ (see Proposition~\ref{P6B}). 
Therefore, by applying \e{n41ab}Ê
with
$$
u=\frac{y}{2}(\py\phi)^2-\frac{y}{2}|\nabla\phi|^2,\quad
f=-y(\py\phi)\nabla\phi, 
$$
we deduce from \e{n48x1} that
\begin{align*}
\int \eta\left(-\mez|\nabla\psi|^2+\mez 
(1+|\nabla\eta|^2)B^2\right)\, dx&=\mez\iint \big[(\py\phi)^2-|\nabla\phi|^2\big]\, dydx\\
&\quad+\frac{h}{2}\int_Q |\nabla\phi|^2(x,-h)\, dx.
\end{align*}
Then, it follows from \e{n48} and \e{n47a} that
\begin{align*}
A&=\left(\frac{d-1}{2}-\alpha\right) \iint \psi G(\eta)\psi\, dxdt+\left(\alpha-\frac{d}{2}\right)\iint g\eta^2\, dxdt\\
&\quad+\alpha \int \eta \psi\, dx\Big\arrowvert_0^T\\
%\left(\frac{1}{4}-\frac{d}{2}\right)\iint \big[\psi G(\eta)\psi+g\eta^2\big]\, dx dt+\frac{1}{4} \int \eta \psi\, dx\Big\arrowvert_0^T\\
&\quad+\frac{1+d-\alpha}{2} \iiint \big[(\py\phi)^2-|\nabla\phi|^2\big]\, dydxdt\\
&\quad+\frac{(1+d-\alpha)h}{2}\iint |\nabla\phi|^2(t,x,-h)\, dxdt.
\end{align*}
We now chose
$$
\alpha=\frac{d}{2}-\uq.
$$
Then the coefficients in front of $\iint \psi G(\eta)\psi\, dxdt$ and $\iint g\eta^2\, dxdt$ coincide. 
On the other hand, notice that, since the energy $\mathcal{H}$ is conserved 
($d\mathcal{H}/dt=0$), one has
$$
\iint \big[\psi G(\eta)\psi+g\eta^2\big]\, dx dt
=T\int \big[\psi G(\eta)\psi+g\eta^2\big]\, dx.
$$
We thus find that
\be\label{n49}
\begin{aligned}
A&=-\frac{T}{4}\iint \big[ \psi G(\eta)\psi +g\eta^2\big]\, dxdt\\
&\quad+\left(\frac{d}{2}-\uq\right)\int \eta \psi\, dx\Big\arrowvert_0^T\\
%\left(\frac{1}{4}-\frac{d}{2}\right)\iint \big[\psi G(\eta)\psi+g\eta^2\big]\, dx dt+\frac{1}{4} \int \eta \psi\, dx\Big\arrowvert_0^T\\
&\quad+\frac{5+2d}{8} \iiint \big[(\py\phi)^2-|\nabla\phi|^2\big]\, dydxdt\\
&\quad+\frac{(5+2d)h}{8}\iint |\nabla\phi|^2(t,x,-h)\, dxdt.
\end{aligned}
\ee

Consequently, to obtain the wanted result \e{n151}, we need only to transform further the sum of last two 
terms in the right-hand side of \e{n49}. To do so, we use again \e{n40} applied with (for some fixed $t$)
$$
u(x,y)=-(y-\eta(t,x))(\py\phi)(t,x,y)^2.
$$
Then $u(x,\eta(t,x))=0$ and $u(x,-h)=0$ and hence $\int_{-h}^{\eta(t,x)} \py u\, dy=0$. 
On the other hand
$$
\py u=-2(y-\eta)\phi_y \phi_{yy}-(\phi_y)^2,
$$
so integrating on $y\in [-h,\eta(x)]$ and then on $x\in [0,L_1]$ we obtain, 
remembering that $\phi_{yy}=-\Delta_x\phi$,
$$
0=\iint u_y=-\iint \phi_y^2 +2\iint(y-\eta)(\py\phi)(\Delta_x\phi).
$$
Since $\nabla\phi\cdot \nu=0$ on $(\partial Q\times [-h,+\infty))\cap \partial\Omega$, 
by integrating by parts we infer that
$$
0=-\iint \phi_y^2 -2\iint(y-\eta)(\py\nabla_x\phi)\cdot\nabla_x\phi+2\iint \phi_y \nabla_x\eta\cdot\nabla_x\phi,
$$
so
\begin{align*}
0&=-\iint \phi_y^2 -\iint \py\big( (y-\eta)|\nabla_x\phi|^2\big)+\iint |\nabla_x\phi|^2+2\iint \phi_y\nabla_x\eta\cdot\nabla_x\phi\\
&=\iint(|\nabla_x\phi|^2-\phi_y^2)-\int (h+\eta)|\nabla_x\phi|^2(x,-h)\, dx+2\iint \phi_y\nabla_x\eta\cdot\nabla_x\phi.
\end{align*}
This proves that
\begin{multline}\label{n100}
\iint(\phi_y^2-|\nabla_x\phi|^2)\, dydx+\int h|\nabla_x\phi|^2(x,-h)\, dx\\
=-\int \eta|\nabla_x\phi|^2(x,-h)\, dx+2\iint \phi_y\nabla_x\eta\cdot\nabla_x\phi \,dydx.
\end{multline}
Plugging this identity in \e{n49}, we complete the proof of the lemma.
\end{proof}
By combining \e{n35} and \e{n151}, we obtain the wanted result~\e{n301}. 
This completes the proof of Theorem~\ref{main}.

\section{Proof of Theorem~\ref{TI}}\label{S:5}

In this section we prove Theorem~\ref{TI}. 
Let $\chi \in C^\infty_0(\xT^d)$ be even in $x_1$ and in $x_2$. 
Consider an 
initial data $(\eta_0,\psi_0)$ such that
$$%\be\label{n305}
\begin{aligned}
\eta_0&\defn\chi(x)\sum_{|n|+|m|\le N} a_{nm}^1\cos\left(\pi \frac{n x_1}{L_1}\right)
\cos\left(\pi\frac{m x_2}{L_2}\right),\\
\psi_0&\defn\chi(x)\sum_{|n|+|m|\le N} a_{nm}^2\cos\left(\pi \frac{n x_1}{L_1}\right)
\cos\left(\pi\frac{m x_2}{L_2}\right).
\end{aligned}
$$
We furthermore assume that the mean value of $\eta_0$ is $0$, which in turn implies that the 
mean value of $\eta(t,\cdot)$ is $0$ for all time. As explained below in Remark~\ref{R:final}, one can consider more general initial data. 

Since these initial data are $C^\infty$ and 
periodic functions defined for $x\in\xR^d$, we can directly solve the Cauchy problem in $\xR^d$ by means of 
Theorem~\ref{T2a}. Let $(\eta,\psi)\in C^\infty([0,T]\times \xT^d)$ be the unique solution of the system~\e{system} with initial data 
$(\eta_0,\psi_0)$ (the fact that $\eta,\psi$ are $C^\infty$ follows from the fact that on can propagate the regularity). 
As explained in Section~\ref{S:2}, since the initial data are even in $x_1,x_2$, one then obtains 
solutions of the equations in the tank by restricting $(\eta,\psi)$. 

In view of Corollary~\ref{C2d2}, to prove Theorem~\ref{TI} it remains only to prove the following result. 
\begin{lemm}\label{L51}
For any $\beta>1/2$, there exist $K_0,c,\kappa$ such that, for any $N\in \xN^*$, 
if
$$
\la a_{nm}^1\ra+\la a_{nm}^2\ra\le cN^{-\kappa},
$$
then the following properties hold with
$$
A=K_0N^\beta.
$$
Firstly, the solution exists on $[0,T(A)]$ with
$$
T(A)=4\left[1+\frac{(2d+3)\max\{L_1,L_2\}}{\sqrt{g}}A\right],
$$
and secondly one has $\eta(t,x)\ge -h/2$ for all $(t,x)\in [0,T]\times \xR^d$ and 
$$
\sup_{t\in [0,T(A)]}\lA \nabla\eta(t)\rA_{L^\infty}\le \frac{1}{5+2d},\quad 
\sup_{t\in[0,T(A)]}\lA \nabla\psi(t)\rA_{L^2}\le A\sqrt{2\mathcal{H}}.
$$
\end{lemm}

To prove this lemma, we shall use 
an interpolation argument and 
the following result about the Cauchy problem on large time intervals.

\begin{theo}[from Lannes~\cite{LannesLivre}]\label{Teps}
Let $s_1>4$. % such that the following result holds. 
For all integer $s\ge s_1$ there exist three positive constants $c_*,\delta_0,C_1$, depending only on $s$  such that, if
\be\label{n120}
\eps_0\defn \lA (\eta_0,\nabla\psi_0)\rA_{H^{s+\mez}(\xT^d)\times H^{s-\mez}(\xT^d)}
\le \delta_0,
\ee
then the solution $(\eta,\psi)$ given by Theorem \ref{T2a} exists on a time interval 
$[0,T]$ for some $T\ge c_*/\eps_0$ and
$$
\lA (\eta(t),\nabla\psi(t))\rA_{H^{s}(\xT^d)\times H^{s-1}(\xT^d)}\le 
C_1 \eps_0.
$$
\end{theo}
\begin{proof}
Let us explain how to obtain this result from Theorem~$4.16$ in \cite{LannesLivre}. 

Set $\eps=M\eps_0$ where $M$ is some large enough real number to be determined later. 
Let $(\eta,\psi)$ be the solution given by Theorem \ref{T2a} and introduce
\be\label{n121}
\eta^\eps(t,x)=\frac{1}{\eps}\eta(t,x),\quad \psi^\eps(t,x)=\frac{1}{\eps}\psi(t,x).
\ee
Then
\begin{equation}\label{system-eps}
\left\{
\begin{aligned}
&\partial_t \eta^\eps=G(\eps\eta^\eps)\psi^\eps,\\
&\partial_t \psi^\eps + g \eta^\eps+ \frac{\eps}{2}|\nabla\psi^\eps|^2
-\frac{\eps}{2(1+|\eps\nabla\eta^\eps|^2)}
\bigl(G(\eps\eta^\eps)\psi^\eps+\eps\nabla\eta^\eps\cdot\nabla\psi^\eps\bigr)^2= 0.
\end{aligned}
\right.
\end{equation}
The proof is based on two facts. Firstly, introduce the energy
$$
\mathcal{E}^\eps(t)=\lA\nabla\psi^\eps(t)\rA_{H^{2}}^2
+\sum_{|\alpha|\le s}\Big(\lA \partial_x^\alpha\eta^\eps(t)\rA_{L^2}
+\lA \nabla\Big(\px^\alpha\psi^\eps-\eps B^\eps\px^\alpha\eta^\eps\Big)(t)\rA_{H^{-\mez}}^2\Big),
$$
where
$$
B^\eps=\frac{G(\eps\eta^\eps)\psi^\eps+\eps\nabla\eta^\eps\cdot\nabla\psi^\eps}{1+|\eps\nabla\eta^\eps|^2}.
$$
Then the first fact we need to know is a consequence of Theorem~$4.16$ in \cite{LannesLivre}. This result 
asserts that, if $s\ge s_1$ and if $\mathcal{E}^\eps(0)$ is small enough, say smaller than $\delta_1$, 
then there is a constant $C_*$ independent of $\eps$ and such that the 
solution exists on a time interval $[0,T^\eps]$ with $T^\eps\ge C_*/\eps$ 
and satisfies
$$
\sup_{t\in [0,T^\eps]} \mathcal{E}^\eps(t)\le 1. 
$$
To introduce the second fact we need to know, we begin by recalling the following technical ingredient: since $3>d/2+1$, 
one has (see \cite[Theorem~$3.15$]{LannesLivre})
\be\label{n127b}
\lA G(\eps\eta^\eps)\psi^\eps\rA_{H^{2}}\le C(\lA \eps\eta^\eps\rA_{H^{3}})\lA \nabla\psi^\eps\rA_{H^{2}}.
\ee
Then, by combining this estimate with the usual nonlinear estimates in Sobolev spaces one obtains that, if 
\be\label{n122b}
\lA (\eta^\eps(t),\nabla\psi^\eps(t))\rA_{H^{3}\times H^{2}}\le 1,
\ee
then $\lA B^\eps(t)\rA_{H^{2}}$ is bounded uniformly in $\eps\in (0,1]$. Then, using the triangle inequality and 
the product rule in Sobolev spaces, it is easily seen that
\be\label{n122c}
c_1\lA (\eta^\eps(t),\nabla\psi^\eps(t))\rA_{H^s\times H^{s-1}}\le \mathcal{E}^\eps(t)\le C_2
\lA (\eta^\eps(t),\nabla\psi^\eps(t))\rA_{H^{s+\mez}\times H^{s-\mez}},
\ee
for some absolute constants $c_1,C_2$. 

We are now in position to prove the wanted result. Fix $M= C_2/\delta_1$. 
Then
$$
\lA (\eta^\eps(0),\nabla\psi^\eps(0))\rA_{H^{s+\mez}\times H^{s-\mez}}\le 
\frac{\eps_0}{\eps}=\frac{1}{M}=\frac{\delta_1}{C_2}.
$$
We can always assume that $\delta_1/C_2\le 1$, so that the condition \e{n122b} is satisfied at time $t=0$. 
This allows us to use \e{n122c} which implies that
$$
\mathcal{E}^\eps(0)\le C_2
\lA (\eta^\eps(0),\nabla\psi^\eps(0))\rA_{H^{s+\mez}\times H^{s-\mez}}\le 
C_2\frac{\delta_1}{C_2}=\delta_1.
$$
As explained above, this implies that the 
solution exists on a time interval $[0,T^\eps]$ with $T^\eps\ge C_*/\eps=(C_*\delta_1/C_2)\eps_0$ 
and satisfies 
$\sup_{t\in [0,T^\eps]} \mathcal{E}^\eps(t)\le 1$. 
Now observe that, by definition of $\mathcal{E}^\eps(t)$, if 
$\mathcal{E}^\eps(t)\le 1$ then the condition \e{n122b} is satisfied. 
This allows us to use the bound \e{n122c} 
to deduce that
\be\label{n128b}
\begin{aligned}
\lA (\eta(t),\nabla\psi(t))\rA_{H^s\times H^{s-1}}
&=\eps\lA (\eta^\eps(t),\nabla\psi^\eps(t))\rA_{H^s\times H^{s-1}}\\
&\leÊ\frac{1}{c_1}\eps
=\frac{1}{c_1}M\eps_0=\frac{C_2}{c_1\delta_1}\eps_0.
\end{aligned}
\ee

This gives the desired result where the 
parameters $c_*,\delta_0,C_1$ in the statement of Theorem~\ref{Teps}Ê
are given by $c_*=C_*\delta_1/C_2$, $\delta_0=\delta_1/C_2$ and $C_1=C_2/(\delta_1 c_1)$.
\end{proof}

We are now in position to prove Lemma~\ref{L51}. 
%Let $s_1$ be as in the statement of 
%Theorem~\ref{Teps}Ê
%and let $\beta>1/2$. 
For some reason that will be clear below, we then fix $s$ such that
$$
s\in \xN,\quad s>4,\quad s\ge \frac{\beta+1/2}{2\beta-1}.
$$
Once $s$ is so fixed, we 
consider $c_*,\delta_0,C_1$ as given by 
the statement of Theorem~\ref{Teps}. 
Recall that we consider 
initial data $(\eta_0,\psi_0)$ such that
\be\label{n305}
\begin{aligned}
\eta_0&\defn\chi(x)\sum_{|n|+|m|\le N} a_{nm}^1\cos\left(\pi \frac{n x_1}{L_1}\right)
\cos\left(\pi\frac{m x_2}{L_2}\right),\\
\psi_0&\defn\chi(x)\sum_{|n|+|m|\le N} a_{nm}^2\cos\left(\pi \frac{n x_1}{L_1}\right)
\cos\left(\pi\frac{m x_2}{L_2}\right).
\end{aligned}
\ee
Directly from the Leibniz' rule, we see that 
for any smooth functions $u_1$, $u_2$, there holds 
$\lA u_1u_2\rA_{H^s}\les \lA u_1\rA_{W^{s,\infty}}\lA u_2\rA_{H^s}$. 
Then, with
$$
H_0\defn \lA (\eta_0,\nabla\psi_0)\rA_{L^{2}(\xT^d)\times H^{-1}(\xT^d)},
$$
one has
$$
\lA (\eta_0,\nabla\psi_0)\rA_{H^{s+\mez}(\xT^d)\times H^{s-\mez}(\xT^d)}
\le C_{2} N^{s+\mez}H_0,
$$
for some constant $C_2$ depending only on $s$ and on 
$\lA \chi\rA_{W^{s,\infty}}$. 
Set
$$
\eps_0\defn C_{2}H_0 N^{s+\mez}.
$$
Let $a\in (0,1]$ be a small positive number to be determined and assume that
\be\label{n132a}
H_0\le \frac{a\delta_0}{C_2N^{s+\mez+\beta}}.
\ee
Then $\eps_0\le a\delta_0/N^\beta$. Consequently, $\eps_0\le \delta_0$ and one can apply Theorem~\ref{Teps} to infer that 
the solution $(\eta,\psi)$ of system \e{system} with initial data 
$(\eta_0,\psi_0)$ 
exists on a time interval $[0,T_*]$ with 
\be\label{n136}
T_*\ge  \frac{c_*}{\eps_0}\ge \frac{c_*}{a\delta_0}N^{\beta},
\ee
and
\be\label{n129b}
\sup_{t\in [0,T_*]} \lA (\eta,\nabla\psi)(t)\rA_{H^{s}(\xT^d)\times H^{s-1}(\xT^d)}
\le C_1\eps_0=C_1C_{2}H_0 N^{s+\mez}.
\ee

Now we use an interpolation inequality in Sobolev spaces: If $\sigma=\lambda \sigma_1+(1-\lambda)\sigma_2$ with 
$\lambda \in [0,1]$, then 
$$
\lA u\rA_{H^{\sigma}}\le C(\sigma_1,\sigma_2)\lA u\rA_{H^{\sigma_1}}^\lambda \lA u\rA_{H^{\sigma_2}}^{1-\lambda}.
$$
This yields
$$
\lA \nabla\psi(t)\rA_{L^2}\le C(-1/2,s-1) \lA\nabla\psi(t)\rA_{H^{-\mez}}^{\lambda(s)}\lA\nabla\psi(t)\rA_{H^{s-1}}^{1-\lambda(s)}\quad\text{with}
\quad \lambda(s)=\frac{s-1}{s-1/2}.
$$
Consequently, it follows from \e{n129b} that
\be\label{n130a}
\lA\nabla\psi(t)\rA_{L^2}\le C(s)N^{\beta(s)}H_0^{1-\lambda(s)}\lA\nabla\psi(t)\rA_{H^{-\mez}}^{\lambda(s)}
\quad\text{with}\quad \beta(s)=\mez \frac{s+1/2}{s-1/2},
\ee
and $C(s)=C(-1/2,s-1) (C_1C_2)^{1-\lambda(s)}$. 

To conclude the proof if remains to estimate $\lA\nabla\psi\rA_{H^{-\mez}}$ 
and $H_0$ in terms of the hamiltonian $\mathcal{H}\defn \mez \int \big[g\eta^2+\psi G(\eta)\psi\big]\, dx$.  
We claim that there exists an absolute constant $K$ such that
\be\label{n131a}
\lA\nabla\psi(t)\rA_{H^{-\mez}}\le K \sqrt{\mathcal{H}}, \quad H_0\le K \sqrt{\mathcal{H}}.
\ee
Let us assume this claim and conclude the proof. 
Set 
$$
K_0\defn KC(s)/\sqrt{2},\quad A\defn K_0N^\beta.
$$
It follows from \e{n130a} and \e{n131a} that
$$
\lA \nabla\psi(t)\rA_{L^2}\le K C(s) N^{\beta(s)}\sqrt{\mathcal{H}} \quad\text{with}\quad \beta(s)=\mez \frac{s+1/2}{s-1/2}.
$$
By definition of $s$, one has $\beta(s)\le\beta$. Therefore, by definition of $A$, 
$$
\lA \nabla\psi(t)\rA_{L^2}\le K C(s) N^{\beta}\sqrt{\mathcal{H}}\le A\sqrt{2\mathcal{H}}.
$$
Eventually, we chose $a$ so small that the lifespan $T_*$ (see \e{n136}) satisfies
$$
T_*\ge \frac{c_*}{a\delta_0}N^{\beta}\ge  
 4\left[1+\frac{(2d+3)\max\{L_1,L_2\}}{\sqrt{g}}A\right]=T(A).
$$
On the other hand, if the initial data are of the form \e{n305} with
 $\la a_{nm}^\ell\ra\le cN^{-\kappa}$, then 
$$
H_0\le c\sqrt{L_1L_2}(2N+1)N^{-\kappa},
$$
so that the condition \e{n132a} holds for $\kappa\ge s+\frac{3}{2}+\beta$ and $c$ small enough. This completes the proof of Lemma~\ref{L51} and hence the proof of Theorem~\ref{TI}.

Therefore, to complete the proof of Theorem~\ref{TI}, it remains only to prove the claim \e{n131a}. 
Notice that, by definition of $H_0$ and $\mathcal{H}$, one has 
$$
H_0\le \lA \eta_0\rA_{L^2}+\lA \nabla \psi_0\rA_{H^{-1/2}},\quad 
\lA \eta_0\rA_{L^2}\le 
\sqrt{\frac{2\mathcal{H}}{g}},
$$
so $H_0\le K \sqrt{\mathcal{H}}$ will be a straightforward consequence of the 
claim $\lA\nabla\psi\rA_{H^{-\mez}}\le K \sqrt{\mathcal{H}}$ at time $t=0$ (with a different constant $K$). 
It is thus sufficient to prove the first inequality in \e{n131a}. To do so, remember that (see Remark~\ref{Rhami}) 
$\sqrt{2\mathcal{H}}$ is larger than the $L^2(\Omega(t))$-norm of $\nabla_{x,y}\phi(t)$ (by definition~\e{n2}). 
Then, one infers a control of the $H^{-1/2}_x$-norm of the trace $\nabla\psi(t,\cdot)$ 
(see \cite[Prop. $3.12$]{LannesLivre}) where the implicit constant depends on $\lA \eta(t)\rA_{H^3}$. 
Now notice that one has a uniform control of $\lA \eta(t)\rA_{H^3}$, namely one has 
$\lA \eta(t)\rA_{H^{3}}\le 1$ as a consequence 
of \e{n129b}, the assumption $s\ge 4$ 
and the fact that one can assume without loss of generality that $C_1\eps_0\le 1$. 
Therefore the $\dot H^{1/2}_x$-norm of the trace $\psi$ is uniformly bounded by 
$K\sqrt{\mathcal{H}}$ for some absolute constant $K$, which implies the desired result 
$\lA\nabla\psi\rA_{H^{-\mez}}\le K \sqrt{\mathcal{H}}$. 

\begin{rema}\label{R:final}
Let $\beta>1/2$ and 
denote by $s$ the smallest integer satisfying
$$
s>4,\quad s\ge \frac{\beta+1/2}{2\beta-1}.
$$
Given $N\in\xN$ and $A,c>0$, we 
denote by $X_N(A,c)$ the set of functions 
$(\eta_0,\psi_0)\in H^{s+1}_e(\xT^d)\times H^{s+1}_e(\xT^d)$ such that 
$$
\eta_0\ge -\frac{h}{2},\quad \int_Q \eta_0 \, dx=0,
$$
and satisfying 
$$
\lA (\eta_0,\nabla\psi_0)\rA_{H^{s+\mez}\times H^{s-\mez}}
\le \min \Big\{A N^{s+\mez}\lA (\eta_0,\nabla\psi_0)\rA_{L^{2}\times H^{-1}}, cN^{-\beta}\Big\}.
$$
Then the previous proof shows that, for any $A>0$, there exists $c>0$ such that, 
for any $N\in \xN$, 
if $(\eta,\psi_0)$ belongs to $X_N(A,c)$ then the conclusion of Theorem~\ref{TI} holds.
\end{rema}

\appendix

\section{Hamiltonian proof}

The proof of Theorem~\ref{main}, which was given in the previous sections, relies on several cancellations. To understand 
these cancellations, we will give in this appendix a (formal) proof of \e{n10} which exploits directly 
the hamiltonian structure of the water-wave equations. 
\footnote{This proof could seem shorter than the one given in Sections~\ref{S:Pohozaev}--\ref{S:4} 
but this is not the case: indeed, it uses some computations 
done in Section~\ref{S:4}, it is not self-contained (we use a shape derivative formula due to Lannes) and also 
we do not try to justify rigorously the computations.} To simplify notations we consider the case $d=1$ and assume that $L_1=1$ (then $\xT=\xR/(2\xZ)$ and we denote by $H^s(\xT)$ the Sobolev space 
of $2$-periodic functions). 

Consider a solution 
$(\eta,\psi)\in C^{0}([0,T];H^{s}_{\pair}(\xT)\times H^{s}_{\pair}(\xT))$ 
of the Craig--Sulem--Zakharov system \e{system} for some $T>0$ and 
$s>5/2$ (remember that the subscript $e$ indicates that we consider even functions). Set
$$
A\defn \int_0^T\int_0^1 \big[(\partial_t\eta)x(\px\psi)-(\partial_t\psi)x(\px\eta)\big]\, dxdt.
$$
We are going to compute $A$ in two different ways, and the wanted identity \e{n10}Ê
will be deduced by comparing the two results. 

%To clarify notations, let us recall that, by definition of $H^s(\xT)$ and since we fix $L_1=1$ in this appendix, the fact that 
%$(\eta,\psi)\in C^{0}([0,T];H^{s}_{\pair}(\xT)\times H^{s}_{\pair}(\xT))$ implies that 
%$\eta$ and $\psi$ are $2$-periodic in $x$. Since, in addition, $\eta,\psi$ are even in $x$, 
%it is enough to handle integrals defined on $[0,1]$ even if the functions are $2$-periodic. 

{\em First computation.} Here is the main new ingredient. In this step we compute $A$ using the 
Hamiltonian structure of the water waves equation. This computation relies on the works by 
Zakharov~\cite{Zakharov1968}, Craig--Sulem~\cite{CrSu} and Lannes~\cite{LannesJAMS}. 

Following Zakharov~\cite{Zakharov1968} and Craig--Sulem~\cite{CrSu}, 
we begin by using the fact that the water-wave system \e{system} can be written as
\be\label{n153}
\partial_t\eta=\frac{\delta \mathcal{H}}{\delta \psi},\quad \partial_t\psi=-\frac{\delta \mathcal{H}}{\delta \eta}
\ee
where\footnote{By assumption, $\eta$ and $\psi$ are $2$-periodic and even in $x$. This is why it is sufficient to consider integrals over $[0,1]$. In particular, 
the bilinear mapping $(u,v)\mapsto \int_0^1 u(x)v(x)\, dx$ is a scalar product on $L^2_e(\xR/(2\xZ))$. 
In \e{n153}, the derivatives of $\mathcal{H}$ 
are understood for this scalar product (instead of $\int_{-1}^1 u(x)v(x)\, dx$), and this is why \e{n153}Ê
holds 
even if $\mathcal{H}$ is only one half of the energy.}
$$
\mathcal{H}=\mez \int_0^1 \big[g\eta^2+\psi G(\eta)\psi\big]\, dx.
$$
%where $G(\eta)\psi$ is as given by Definition~\ref{D:8}. The fact that $\eta,\psi$ satisfy \e{n153} 
%is a well-known result for periodic water waves. That the same property holds in a rectangular tank follows from the arguments used in Section~\ref{S:2}. Alternatively, one can consider only $2$-periodic functions which are even in $x$ -- then by symmetry it is enough to consider integrals on $[0,1]$. 

The key point is the following: we have
$$
A=\int_0^T a(t)\, dt\quad \text{with}\quad a\defn \int_0^1 \left[x\psi_x\frac{\delta \mathcal{H}}{\delta \psi}
+x\eta_x\frac{\delta \mathcal{H}}{\delta \eta}\right]\, dx,
$$
which means that
$$
a=\lim_{\eps\rightarrow 0}\frac{1}{\eps}\left[ \mathcal{H}(\eta,\psi+\eps x\psi_x)-\mathcal{H}(\eta,\psi)\right]
+\lim_{\eps\rightarrow 0}\frac{1}{\eps}\left[ \mathcal{H}(\eta+\eps x\eta_x,\psi)-\mathcal{H}(\eta,\psi)\right].
$$
Since $G(\eta)$ is self-adjoint, one has immediately
$$
\lim_{\eps\rightarrow 0}\frac{1}{\eps}
\left[ \mathcal{H}(\eta,\psi+\eps x\psi_x)-\mathcal{H}(\eta,\psi)\right]=\int (G(\eta)\psi)(x\psi_x)\, dx.
$$
On the other hand
$$
\lim_{\eps\rightarrow 0}\frac{1}{2\eps}\int\left[ (\eta+\eps x\eta_x)^2-\eta^2\right]\, dx=\int \eta x\eta_x\, dx
$$
so the only difficulty is to compute
$$
\lim_{\eps\rightarrow 0}\frac{1}{2\eps}\int\left[\psi G(\eta+\eps x\eta_x)\psi-\psi G(\eta)\psi\right]\, dx.
$$
To do so, we use a formula due to Lannes (see \cite{LannesJAMS,LannesLivre}) which allows us 
to compute the derivative of $G(\eta)\psi$ with respect to $\eta$. This formula gives
$$
G(\eta+\eps x\eta_x)\psi=G(\eta)\psi-\eps G(\eta)\big(B x\eta_x\big)-\eps \px (Vx\eta_x)+O(\eps^2),
$$
where, as above, $B=\py\phi\arrowvert_{y=\eta}$ and $V=\px\phi\arrowvert_{y=\eta}$. 
Using again the fact that $G(\eta)$ is self-adjoint, we obtain
\begin{multline*}
\lim_{\eps\rightarrow 0}\frac{1}{2\eps}\int \left[\psi G(\eta+\eps x\eta_x)\psi-\psi G(\eta)\psi\right]\, dx\\
=-\mez\int \left[ (G(\eta)\psi)(Bx\eta_x)+ \psi\px (Vx\eta_x)\right]\, dx.
\end{multline*}
By combining the previous results, we get
$$
a=g\int \eta x\eta_x\, dx+\int (G(\eta)\psi)(x\psi_x)\, dx-\mez\int \left[(G(\eta)\psi)(Bx\eta_x)+ \psi\px (Vx\eta_x)\right]\, dx.
$$
Since $V(t,x)=0$ for $x=1$ (by assumption) one has
$$
-\mez \int \psi\px (Vx\eta_x)\, dx=\mez \int (x\psi_x)V\eta_x\, dx.
$$
Then, gathering the terms in a different way, we find that
$$
a=g\int \eta x\eta_x\, dx+\mez \int (G(\eta)\psi)x(\psi_x-B\eta_x)\, dx
+\mez\int \left[(G(\eta)\psi+V\eta_x)(x\psi_x)\right]\, dx.
$$
Remembering that
$$
\psi_x-B\eta_x=V,\quad G(\eta)\psi+V\eta_x=B,
$$
we get
$$
a=g\int \eta x\eta_x\, dx+\mez \int (G(\eta)\psi)xV\, dx
+\mez\int B x\psi_x\, dx.
$$
Recalling that $m(t)=\eta(t,1)$ and computing the first term, we conclude that
\be\label{n87}
A=\frac{g}{2} \int m(t)^2\, dt-\frac{g}{2} \iint \eta^2\, dx dt +\mez \iint \Big[(G(\eta)\psi)xV + B x\psi_x\Big]\, dxdt.
\ee

{\em Second computation.} We now compute $A$ using integration by parts and the equations for $\eta$ and $\psi$. 
This second computation is not new. However, since 
it does not appear explicitly in the previous section, for the sake of readability, we redo the analysis. 

Directly from the definition of $A$, by integration by parts, one has
$$
A=-\int \eta(t,1)\partial_t\psi(t,1)\, dt+\int \eta x\psi_x\, dx\Big\arrowvert_{t=0}^{t=T}+\iint \eta\partial_t\psi \, dxdt.
$$
It is convenient to split the last term as
$$
\iint \eta\partial_t\psi \, dxdt=\uq \iint \eta\partial_t\psi \, dxdt+\frac{3}{4}\iint \eta\partial_t\psi \, dxdt.
$$
To handle the first term in the right-hand side, we integrate by parts in time 
and replace $\partial_t\eta$ by $G(\eta)\psi$. For the second term, we 
use the equation for $\psi$ written under the form (see~\e{n25})
$$
\partial_t\psi=-g\eta-\mez V^2 -B V\partial_x \eta+ \mez B^2.
$$
Then one gets
\begin{align*}
\iint \eta\partial_t\psi \, dxdt&=-\uq \iint \psi G(\eta)\psi\, dxdt+\uq\int \eta\psi\, dx\Big\arrowvert_{t=0}^{t=T}\\
&\quad -\frac{3g}{4}\iint \eta^2\, dxdt-\frac{3}{8}\iint \eta\left(V^2 +2B V\partial_x \eta-  B^2\right)\,dxdt.
\end{align*}
Recalling that
$\partial_t\psi(t,1)=-gm(t)+\mez m'(t)^2$ (see \e{n94}), we end up with
\be\label{n89}
\begin{aligned}
A&=\int\big[gm(t)^2-\mez m(t)m'(t)^2\big]\,dt\\
&\quad -\uq \iint \psi G(\eta)\psi\, dxdt-\frac{3g}{4}\iint \eta^2\, dxdt\\
&\quad+\uq\int \eta\psi\, dx\Big\arrowvert_{t=0}^{t=T}+\int x\eta\px\psi\, dx\Big\arrowvert_{t=0}^{t=T}\\
&\quad-\frac{3}{8} \iint \eta\left(V^2 +2B V\partial_x \eta-  B^2\right)\,dxdt.
\end{aligned}
\ee

{\em Conclusion.}Ê
By combining \e{n87} and \e{n89} we obtain that
\be\label{n90}
\begin{aligned}
\mez\int\big[gm(t)^2-m(t)m'(t)^2\big]\,dt&=\uq \iint \psi G(\eta)\psi\, dx dt +\frac{g}{4}\iint \eta^2\, dxdt\\
&\quad-\uq\int \eta\psi\, dx\Big\arrowvert_{t=0}^{t=T}-\int x\eta\px\psi\, dx\Big\arrowvert_{t=0}^{t=T}\\
&\quad+R
\end{aligned}
\ee
with
$$
R\defn \iint \left[ \frac{3}{8}\eta\left( V^2 +2B V\partial_x \eta-  B^2\right)+\mez  (G(\eta)\psi)xV 
+\mez B (x\psi_x)\right]\, dxdt.
$$
Since 
$$
B=\py\phi \arrowvert_{y=\eta},\quad 
V=\px\phi \arrowvert_{y=\eta},\quad 
G(\eta)\psi=(\py\phi -\eta_x\px\phi) \arrowvert_{y=\eta},
$$
and since $\psi_x=(\px\phi +\eta_x\py\phi) \arrowvert_{y=\eta}$, 
one can write $R$ as
$$
R=\iint u(t,x,\eta(t,x))\, dx dt+\iint f(t,x,\eta(t,x))\eta_x(t,x)\, dx dt
$$
where
\begin{align*}
u(x,y)&=\frac{3}{8} y (\px\phi)^2-\frac{3}{8}y(\py\phi)^2+ x(\px\phi)(\py\phi),\\
f(x,y)&=\frac{3}{4}y(\px\phi)(\py\phi)-\mez x(\px\phi)^2+\mez x(\py\phi)^2.
\end{align*}
Then one has 
$$
\py u-\px f=\frac{7}{8} \big((\px\phi)^2-(\py\phi)^2\big).
$$
Consequently, it follows from 
the identities \e{n41ab} and \e{n100} that
$$
R=\iint \left(\frac{h}{2}+\frac{7}{8}\eta\right)(\px\phi)(t,x,-h)^2\, dxdt-\frac{7}{4}\iiint (\px\eta)(\px\phi)(\py\phi)\,dydxdt.
$$
The wanted identity \e{n10} thus follows from \e{n90}.

\bibliographystyle{plain}
\bibliography{bib_Boundary}

%\vspace{1cm}

%\clearpage

Thomas Alazard\\[1ex]
CNRS et D\'epartement de Math\'ematiques et Applications UMR 8553\\
\'Ecole normale sup\'erieure \\
45 rue d'Ulm\\
Paris F-75005, France

\end{document}